\documentclass[12pt,a4paper]{amsart}

\usepackage{mathrsfs}
\usepackage{amssymb,amsmath,mathtools,amsthm,color}
\usepackage{graphicx,cite,txfonts}
\usepackage{url}
\usepackage{setspace}
\usepackage{enumerate}
\usepackage[headheight=110pt,top=1.4in,bottom=1.4in,left=1.0in,right=1.0in]{geometry}
\usepackage{tikz}
\usepackage{tikz-cd}
\usetikzlibrary{matrix}
\usepackage{hyperref}
\hypersetup{colorlinks,linkcolor={red},citecolor={red},urlcolor={blue}}

\newtheorem{theorem}{Theorem}[section]

\newtheorem{proof of lemma}[theorem]{Proof of Lemma}
\newtheorem{proposition}[theorem]{Proposition}

\newtheorem{corollary}[theorem]{Corollary}

\theoremstyle{definition}
\newtheorem{definition}[theorem]{Definition}

\newtheorem{remark}[theorem]{Remark}

\numberwithin{equation}{section}
\makeatletter
\@namedef{subjclassname@2020}{\textup{2020} Mathematics Subject Classification}
\makeatother

\newcommand{\bigslant}[2]{{\raisebox{.2em}{$#1$}\left/\raisebox{-.2em}{$#2$}\right.}}

\newcommand{\Lip}{\operatorname{Lip}}

\newcommand{\BLip}{\operatorname{BLip}}
\newcommand{\BLipz}{\operatorname{BLip}_0}
\newcommand{\Blin}{\operatorname{Blin}}
\newcommand{\Fix}{\operatorname{Fix}}
\newcommand{\spancl}{\overline{\operatorname{span}}}
\newcommand{\F}{\mathcal{F}}

\newcommand{\whot}{\widehat{\otimes}_{\pi}}

\title{Group Actions, Fixed Points and Orbit Quotients in Lipschitz Spaces}

\author[Arindam Mandal]{Arindam Mandal$^\dagger$}
	\address{Arindam Mandal, School of Mathematical Sciences, National Institute of Science Education and Research
		Bhubaneswar, An OCC of Homi Bhabha National Institute, P.O. Jatni, Khurda, Odisha 752050,
		India.}
	\email{arindam.mandal@niser.ac.in}

\thanks{$\dagger$ \texttt{Corresponding author}}

\date{}

\begin{document}

\baselineskip=17pt
\begin{abstract}
We study Lipschitz mappings associated with group actions by similarities, extending the orbit-quotient approach for point-preserving isometric actions. Normalizing by the similarity ratios we obtain isometric representations whose fixed points are character-equivariant Lipschitz mappings. The corresponding quotients by closed spans of orbit differences provide canonical preduals and norm-preserving linearizations, without an amenability assumption. 

Applications include positively homogeneous, linear, and bilinear mappings. In particular, we prove that the space of positively homogeneous Lipschitz functions form a $1$-complemented subspace of the Lipschitz space (strenghthening the earlier existed result) and interpret existing projections onto linear and bilinear mappings as invariant-mean averages. For amenable groups, the same approach yields isometric realizations of the quotient spaces in the biduals of the original linearization spaces.
\end{abstract}

\maketitle

%==============
\section{Introduction}
The motivation of this work comes from the interplay between the algebraic structure of group actions and the analytic properties of Lipschitz spaces. A Lipschitz map between metric spaces, is a map for which the distance between any two images is bounded by a fixed constant times the distance between the corresponding points.
Lipschitz theory (in particular, Lipschitz-free space) provides a linear framework for studying
the geometry of metric spaces and the structure of Lipschitz
mappings. 

For a pointed metric space $(M,d,0)$, the space
$\Lip_0(M)$ of real-valued Lipschitz functions vanishing at
the distinguished point is canonically the dual of the
Lipschitz-free space $\F(M)$. Every origin-preserving
Lipschitz mapping $f \in \Lip_0(M,E)$ from $M$ into a Banach space $E$ admits a
unique bounded linear extension $\widehat{f} \in \mathcal{L}(\F(M),E)$ (called as linearization of the Lipschitz map $f$) through the canonical
embedding $\delta_M\colon M\to\F(M);~~x \mapsto \delta_x \in \Lip_0(M)^*$, with preservation of
the norm. This correspondence allows conditions on nonlinear
mappings to be expressed as linear relations in $\F(M)$;
see \cite{LFBS,LA} for background. Here, we note that the Lipschitz space
$\Lip_0(M,E)$ forms a Banach space with respect to the Lipschitz norm
$$
\Lip(f)=\sup_{\substack{x,y\in M\\x\neq y}}\frac{\|f(x)-f(y)\|}{d(x,y)}.$$

Group actions give a natural source of such relations.
For an isometric action on $M$, invariance of a Lipschitz
function is expressed by the vanishing of its linearization
on the vectors $\delta_{gx}-\delta_x$. C\'uth and Doucha
\cite{PILFSIBGA} developed this observation in their study of
projections induced by group actions. They identified the
free space over the metric orbit quotient with a quotient
of $\F(M)$ by the closed span of orbit differences (see \cite[Lemma~4.3]{PILFSIBGA}), and
used amenable averaging to obtain bidual realizations and,
under additional hypotheses, complemented embeddings (see \cite[Theorem~4.5]{PILFSIBGA}).
Their results connect the geometry of metric orbit spaces
with the linear structure of Lipschitz-free spaces.

Building on the classical notions of equivariance, invariants, and coinvariants, we develop a group-action framework for studying distinguished spaces of Lipschitz mappings and their associated linearization spaces, with applications to complementability and bidual realizations. Our notion of equivariance follows the standard formulation for mappings between spaces equipped with group actions. In the Lipschitz setting, this formulation appears, for example, in Gu\'eritaud and Kassel's \emph{Maximally stretched laminations on geometrically finite hyperbolic manifolds} \cite[p.~694]{MSLOGFHM}, where equivariance is defined relative to prescribed representations on the domain and codomain.

Starting from the isometric group actions considered by C\'uth and Doucha \cite{PILFSIBGA}, we consider point-preserving actions that also allow changes of scale. More precisely, we replace isometries by similarities, that is, bijections that multiply all distances by a fixed positive factor. If a group $G$ acts by similarities on a nontrivial pointed metric space $(M,d,0)$ and fixes $0$, then each $g\in G$ determines a unique $\chi(g)>0$ such that
$d(gx,gy)=\chi(g)d(x,y),$ for any $ x,y\in M$.
Composition gives $\chi(gh)=\chi(g)\chi(h)$, so these scaling factors form a positive character $\chi\colon G\to(0,\infty)$, called the similarity character of the action. The choice $\chi\equiv1$ recovers the point-preserving isometric case of \cite{PILFSIBGA}. A basic example beyond that case is the dilation action of $(0,\infty)$ on a nonzero Banach space, for which $\chi(r)=r$.

For each $g\in G$, the number $\chi(g)$ is the similarity ratio of the map $x\mapsto gx$. To reflect this scaling in the codomain, we let $G$ act on $E$ by $g\cdot u=\chi(g)u$. Since $\chi$ is a character, this defines the scalar representation $g\mapsto\chi(g)I_E$. Equivariance with respect to these domain and codomain actions is therefore expressed by
$f(gx)=\chi(g)f(x)$, for any $ g\in G,\ x\in M$,
which motivates Definition~\ref{def:chi-equivariant}.

These actions also induce operators on the space of Lipschitz mappings. For each $g\in G$, define
$U_g\colon\Lip_0(M,E)\longrightarrow\Lip_0(M,E)$ such that $
(U_gf)(x)=\chi(g)f(g^{-1}x)$.
Here precomposition by $g^{-1}$ scales distances in the domain by $\chi(g)^{-1}$, while multiplication by $\chi(g)$ compensates for this change. Consequently, $\Lip(U_gf)=\Lip(f)$, and the family $(U_g)_{g\in G}$ forms a representation by surjective linear isometries. Moreover, $U_gf=f$ for every $g\in G$ precisely when $f$ is character-equivariant. Thus the equivariance condition becomes a fixed-point condition on the Lipschitz space.

The quotient construction has a parallel algebraic
motivation. For a representation of $G$ on a vector
space $X$, the space of coinvariants is obtained by
factoring out the linear span of the differences
$gv-v$. We refer to Tom Dieck's
\emph{Representation Theory}
\cite[\S\S7.6.3--7.6.4]{tomDieck} for invariants and
coinvariants, and to Meir's
\emph{Universal Rings of Invariants}
\cite[discussion preceding Lemma~2.2]{UROI}
for the same coinvariant construction.
We use its norm-closed version for the normalized
representation on the Lipschitz-free space. Taking
the norm closure ensures that the resulting quotient
is a Banach space; its dual is precisely the space
of character-equivariant Lipschitz functions.

These constructions therefore specialize standard
algebraic notions to the metric and Banach space
settings. The underlying group-theoretic terminology,
similarity actions, and character-equivariance are
developed in
Subsections~\ref{subsec: group theo notion},
\ref{subsec:act by sim}, and
\ref{subsec: cha equ Lip map}, respectively.
The subsequent sections establish their linearization
properties and apply them to positively homogeneous,
linear, and bilinear mappings.
We then consider recentered
translations, whose fixed points are bounded linear
mappings, and their two-variable counterparts, whose fixed
points are bounded bilinear mappings. The common mechanism
is the passage from an isometric representation on a
linearization space to its dual fixed-point space and
its Banach space of coinvariants.

\noindent\textbf{Notation.}
Unless otherwise stated, $(M,d,0)$ denotes a pointed metric space, and all Banach spaces $X,Y,E$ are over the real field. By $X^*$ we consider the Banach dual of $X$. We denote by $\ell_\infty(M)$ the Banach space of bounded real-valued functions on $M$, equipped with the supremum norm. The groups of surjective isometries of $M$ and surjective linear isometries of $X$ are denoted by $\operatorname{Isom}(M)$ and $\operatorname{Isom}_{\mathrm{lin}}(X)$, respectively. Group actions on pointed metric spaces are assumed to fix the distinguished point unless explicitly stated otherwise. The symbol $\cong, \equiv$ denotes a linear isometric isomorphism and equivalence, respectively. We write $X\widehat{\otimes}_\pi Y$ for the completed projective tensor product of $X$ and $Y$, and $\mathcal L(X,Y)$ for the Banach space of bounded linear operators from $X$ to $Y$, equipped with the operator norm. For the basics on tensor product theory we refer \cite{RAR} Finally, $I_X$ denotes the identity operator on $X$, and $GL(X)$ denotes the group of bounded linear bijections from $X$ onto itself with bounded inverses. For any abelian group $G$, $~m_r$ means applying the invariant mean $m \in \ell_{\infty}(G)^{*}$ to a bounded function of the variable $r$. For details on invariant mean we refer \cite{GNFA}. Further notation is introduced as needed.
\subsection{Main results}
\label{subsec:main-results}
We now state the main results of this article; their detailed formulations and proofs are given in the subsequent sections, beginning
with the linearization theorem for character-equivariant
Lipschitz mappings.

Let a group $G$ acts on a nontrivial pointed metric space
$(M,d,0)$ by similarities fixing $0$, and let
$\chi\colon G\to(0,\infty)$ be the associated
similarity character. Thus
$d(gx,gy)=\chi(g)d(x,y)$ for all $g\in G,\ x,y\in M$.
For a real Banach space $E$, we consider
\[
\Lip_0^{G,\chi}(M,E)
=
\left\{
f\in\Lip_0(M,E):
f(gx)=\chi(g)f(x)
\text{ for all }g\in G,\ x\in M
\right\}.
\]
In the scalar-valued case, we abbreviate this space
to $\Lip_0^{G,\chi}(M)$.

Again, we define
$\mathcal Y_{G,\chi}
=
\overline{\operatorname{span}}
\{\delta_{gx}-\chi(g)\delta_x:
g\in G,\ x\in M\}
\subseteq\F(M)$,
where the closure is taken in norm, and set
$\F^{G,\chi}(M)=\F(M)/\mathcal Y_{G,\chi}$.
Let $q_{G,\chi}\colon\F(M)\to\F^{G,\chi}(M)$
denote the quotient map.

\begin{theorem}\label{thm:main-equivariant}
With the preceding notation, the following assertions hold.
\begin{enumerate}
\item[(i)] For each $g\in G$, the formula
$(U_gf)(x)=\chi(g)f(g^{-1}x)$, for $f\in\Lip_0(M),\ x\in M$,
defines a surjective linear isometry $U_g\colon\Lip_0(M)\to\Lip_0(M)$. There also exists a unique surjective linear isometry $V_g\colon\F(M)\to\F(M)$ satisfying
$V_g\delta_x=\chi(g)^{-1}\delta_{gx}$, for any $ x\in M$.
These operators satisfy
$U_e=I_{\Lip_0(M)}, ~V_e=I_{\F(M)},~ U_gU_h=U_{gh}, ~V_gV_h=V_{gh}$
for all $g,h\in G$, where $e$ is the identity of $G$. Under the canonical identification $\F(M)^*=\Lip_0(M)$,
for any $g\in G$, we have
$$U_g=V_{g^{-1}}^* ~\mbox{and}~\Fix(U):=\{f\in\Lip_0(M):U_gf=f\text{ for every }g\in G\}=\Lip_0^{G,\chi}(M).$$

\item[(ii)] The subspace $\mathcal Y_{G,\chi}$ is precisely the \emph{orbit-difference subspace} associated with $V=(V_g)_{g \in G}$ satisfies
$\mathcal Y_{G,\chi}=\overline{\operatorname{span}}\{V_g\mu-\mu:g\in G,\ \mu\in\F(M)\}$,
where the closure is taken in the norm of $\F(M)$. Its annihilator is
\[
\mathcal Y_{G,\chi}^{\perp}:=\{f\in\Lip_0(M):\langle\mu,f\rangle=0\text{ for every }\mu\in\mathcal Y_{G,\chi}\}=\Lip_0^{G,\chi}(M).
\]
In particular, the adjoint of the quotient map,
$q_{G,\chi}^*\colon\bigl(\F^{G,\chi}(M)\bigr)^*\longrightarrow\Lip_0(M)$,
is a linear isometry with range $\Lip_0^{G,\chi}(M)$, given explicitly by
\[
(q_{G,\chi}^*\varphi)(x)=\varphi\bigl(q_{G,\chi}(\delta_x)\bigr), \qquad \varphi\in\bigl(\F^{G,\chi}(M)\bigr)^*,\ x\in M.
\]

\item[(iii)] The orbit quotient $\F^{G,\chi}(M)$ has the following universal property. For every real Banach space $E$ and every $f\in\Lip_0^{G,\chi}(M,E)$, there exists a unique bounded linear operator $\widetilde f\colon\F^{G,\chi}(M)\to E$ such that
$\widetilde f\bigl(q_{G,\chi}(\delta_x)\bigr)=f(x)$ for any $ x\in M$.
If $\widehat f\colon\F(M)\to E$ denotes the canonical linearization of $f$, then
\[
\widehat f=\widetilde f\circ q_{G,\chi}, \qquad \|\widetilde f\|=\|\widehat f\|=\Lip(f).
\]
Moreover, the map
$\Lambda_E\colon\Lip_0^{G,\chi}(M,E)\longrightarrow\mathcal L(\F^{G,\chi}(M),E),~~\Lambda_E(f)=\widetilde f$,
is a surjective linear isometry, whose inverse is
$\Lambda_E^{-1}(T)=T\circ q_{G,\chi}\circ\delta_M, $ for any $T\in\mathcal L(\F^{G,\chi}(M),E)$.
\end{enumerate}
\end{theorem}

The normalization by $\chi$ compensates for the change
of scale in the original metric action and produces
isometric representations on the associated Banach
spaces. The theorem then identifies
$\F^{G,\chi}(M)$ as a canonical predual of the
scalar-valued character equivariant Lipschitz space and as
a norm-preserving linearization space for its
Banach-valued counterpart.

When $\chi\equiv1$, the defining relations reduce
to $\delta_{gx}-\delta_x$, and the construction
recovers the point-preserving isometric case of
the metric orbit-quotient framework in
\cite{PILFSIBGA}. No amenability assumption is
needed for Theorem~\ref{thm:main-equivariant};
amenability enters subsequently in the construction
of averaging projections and bidual realizations.

One of the main applications concerns positively homogeneous
Lipschitz mappings on a real Banach space $X$. The
multiplicative group $(0,\infty)$ acts on $X$ by
dilations, with similarity character $\chi(r)=r$.
Consequently,
\[
\Lip_0^{ph}(X)
=
\Fix\left(
f\longmapsto r f(\,\cdot\,/r)
\right)_{r>0}, ~\mbox{and}~
\F^{ph}(X)
\cong
\F(X)\big/
\overline{\operatorname{span}}
\{\delta_{rx}-r\delta_x:r>0,\ x\in X\}.
\]
This identifies the positively homogeneous free space
studied in \cite{AM2} as a space of coinvariants
for the normalized dilation representation.

Amenability of $(0,\infty)$ further yields a norm-one
projection
$P_{ph}\colon\Lip_0(X)\longrightarrow\Lip_0^{ph}(X)$
for every nonzero real Banach space $X$. Explicitly,
$(P_{ph}f)(x)
=
m_r\left[r f\left(\frac{x}{r}\right)\right]$
where $m$ is an invariant mean on
$\ell_\infty((0,\infty))$.
This strengthens the complementability result of
\cite[Proposition~4.2.5]{MandalThesis}, where a radial
extension projection was constructed with norm at most
$3$. The averaging projection is distinct from
the radial extension projection.

The translation applications require a different
normalization because translations do not preserve
the distinguished point. On $\Lip_0(X,E)$, we use
the recentered operators
$(\rho_af)(x)=f(x+a)-f(a)$.
Their fixed points are precisely the bounded linear
mappings. On $\F(X)$, the corresponding operators are
$S_a\delta_x=\delta_{x+a}-\delta_a$
and their orbit differences generate the kernel of
the barycentre (linear) map $\beta_X$, a linear left inverse of $\delta_X$:
\[
\overline{\operatorname{span}}
\{\delta_{x+a}-\delta_x-\delta_a:x,a\in X\}
=
\ker(\beta_X).
\]
The resulting coinvariant space is therefore
canonically isometric to $X$. In this formulation,
the projection onto $\mathcal L(X,E)$ constructed
in \cite{AM3}, for dual ranges $E$, is invariant-mean
averaging of the recentered translation orbit.

The same construction extends to two-Lipschitz
mappings vanishing on the coordinate axes. For real
Banach spaces $X$, $Y$, and $E$, the representation
\[
(\rho_{(a,b)}T)(x,y)
=T(x+a,y+b)-T(x+a,b)-T(a,y+b)+T(a,b)
\]
acts isometrically on $\BLipz(X,Y;E)$, and
$\Fix(\rho)=\Blin(X,Y;E)$.
Its predual counterpart acts on
$\F(X)\widehat\otimes_\pi\F(Y)$.
We identify the closed span $\mathcal Z_{X,Y}$ of
its orbit differences as
\[
\mathcal Z_{X,Y}
=
\ker(\beta_X\widehat\otimes_\pi\beta_Y),~\mbox{and hence obtain}~
\bigl(\F(X)\widehat\otimes_\pi\F(Y)\bigr)/
\mathcal Z_{X,Y}
\cong
X\widehat\otimes_\pi Y
\]
linearly isometrically. This gives a translation-based
description of the projective tensor product and
interprets the projection of \cite{AM1} as averaging
onto the bilinear fixed-point space.

Finally, we apply the bidual averaging construction
of C\'uth and Doucha \cite[Lemma~3.2]{PILFSIBGA} to these
representations. For an amenable discrete group acting
isometrically on a Banach space $X$, the corresponding
coinvariant space admits an isometric embedding into
$X^{**}$, and the dual fixed-point space is contractively
complemented in $X^*$. This gives, in particular,
isometric embeddings
\[
\F^{ph}(X)\longrightarrow\F(X)^{**},
\qquad
X\widehat\otimes_\pi Y
\longrightarrow
\bigl(\F(X)\widehat\otimes_\pi\F(Y)\bigr)^{**}.
\]
For strongly continuous compact-group representations,
Haar averaging takes values in the original space and
identifies the coinvariant space with a contractively
complemented fixed-point subspace. Throughout, we
distinguish these complemented realizations from the
quotient identifications, which hold independently
of amenability.
%================================================
\section{Preliminaries}
\subsection{Group-theoretic notions} \label{subsec: group theo notion}

We begin by recalling some standard terminology concerning group actions;
see \cite[Definition~2.2.1]{DAS}, \cite{tomDieck}. Let $G$ be a group with identity
element $e$, and let $S$ be a nonempty set. An action of $G$ on $S$ is
determined by a group homomorphism
$\rho\colon G\longrightarrow \operatorname{Sym}(S)$,
where $\operatorname{Sym}(S)$ denotes the group of all bijections of
$S$. As usual, we write
$gx:=\rho(g)(x)$, for all $g\in G,\ x\in S$.
A set equipped with such an action is called a $G$-set. If $M$ and $N$
are $G$-sets, a map $f\colon M\to N$ is said to be $G$-equivariant if for all $g\in G,\ x\in M, ~ ~f(gx)=gf(x)$.

In particular, suppose that $M$ and $N$ are metric spaces and that $G$
acts on them through homomorphisms
\[
\alpha\colon G\longrightarrow \operatorname{Isom}(M)
\quad\text{and}\quad
\beta\colon G\longrightarrow \operatorname{Isom}(N).
\]
Then a map $f\colon M\to N$ is $(\alpha,\beta)$-equivariant precisely
when
$f(\alpha(g)x)=\beta(g)f(x)$ for all $ g\in G,\ x\in M$.
This point of view is standard in the study of equivariant Lipschitz
mappings; see, for example, \cite[p.~694]{MSLOGFHM}.

For our purposes, we shall focus on a particular class of group actions,
namely actions by similarities. The preceding formulation of
$(\alpha,\beta)$-equivariance then provides a natural motivation for the
notion of a \emph{character-equivariant Lipschitz mapping}, introduced
in Definition~\ref{def:chi-equivariant}. As will be seen throughout the paper,
this notion provides a common framework for several distinguished
classes of Lipschitz mappings.

\subsection{Group actions by similarities} \label{subsec:act by sim}

Let $(M,d)$ be a nontrivial metric space. We denote by
$\operatorname{Sim}(M)$ the group of all similarities of $M$, that is,
\[
\operatorname{Sim}(M)
:=
\left\{
T\colon M\to M \text{ bijective} :
d(Tx,Ty)=c_Td(x,y)
\text{ for some }c_T>0
\right\}.
\]
Since $M$ is nontrivial, the constant $c_T$ is uniquely determined by
$T$ and will be referred to as the \emph{similarity ratio} of $T$. The map
$\lambda\colon \operatorname{Sim}(M)\longrightarrow (0,\infty),~
\lambda(T)=c_T$,
is a group homomorphism. Indeed, for $S,T\in\operatorname{Sim}(M)$ and
$x,y\in M, ~ ~
d(STx,STy)
=
c_S d(Tx,Ty)
=
c_Sc_Td(x,y),$
and hence
\[
\lambda(ST)=\lambda(S)\lambda(T).
\]

Suppose now that $G$ acts on $M$ by similarities; equivalently, the
action is given by a homomorphism
$
\rho\colon G\longrightarrow \operatorname{Sim}(M)$.
The composition
$\chi:=\lambda\circ\rho\colon G\longrightarrow(0,\infty)$
is then a positive character of $G$, and $d(gx,gy)=\chi(g)d(x,y),$ for all $ g\in G,\ x,y\in M$.
We shall refer to $\chi$ as the \emph{similarity character} associated
with the action.

The usual isometric actions occur as the special case
$\chi\equiv 1$. Indeed, in this case $d(gx,gy)=d(x,y),$ for all $ g\in G,\ x,y\in M$,
so that the action takes values in $\operatorname{Isom}(M)$.

The preceding discussion may be summarized as follows.

\begin{proposition}\label{prop:similarity-character}
Let $(M,d)$ be a nontrivial metric space and let $G$ act on $M$. Then
the following assertions are equivalent:
\begin{itemize}
\item[(i)] the action of $G$ on $M$ is by similarities;
\item[(ii)] there exists a positive character
$\chi\colon G\longrightarrow(0,\infty)$
such that
\[
d(gx,gy)=\chi(g)d(x,y), ~\mbox{for any}~  g\in G,\ x,y\in M.
\]
\end{itemize}
Moreover, whenever these conditions hold, the character $\chi$ is
uniquely determined by the action.
\end{proposition}

\begin{proof}
Let the action is by similarities, and suppose
$\rho\colon G\to\operatorname{Sim}(M)$ denotes the corresponding
homomorphism. Then $\chi=\lambda\circ\rho$ is a positive character and,
by the definition of the similarity ratio, $d(gx,gy)=\chi(g)d(x,y)$ for every $g\in G$ and $x,y\in M$.

Conversely, suppose that such a positive character $\chi$ exists. Then for
each $g\in G$, the map $x\mapsto gx$ is a bijection and satisfies
$d(gx,gy)=\chi(g)d(x,y),$ for all $x,y\in M$.
Thus $g$ acts as a similarity with similarity ratio $\chi(g)$.
Uniqueness of $\chi$ follows from the assumption that $M$ contains at
least two distinct points.
\end{proof}

\noindent\textbf{Convention.}
Throughout the remainder of the paper, unless explicitly stated otherwise,
whenever a positive character $\chi\colon G\to(0,\infty)$ is considered
in connection with an action of $G$ on a metric space $(M,d)$, it is
understood that the action is by similarities and that $\chi$ is the
associated similarity character; that is,
$d(gx,gy)=\chi(g)d(x,y)$ for all $g\in G,~x,y\in M$.

\subsection{Character-equivariant Lipschitz mappings} \label{subsec: cha equ Lip map}

Let $(M,d,0)$ be a nontrivial pointed metric space and let $G$ act on
$M$ by similarities, fixing the distinguished point $0$. By
$\chi\colon G\longrightarrow(0,\infty)$ we denote the associated similarity character. Thus
$d(gx,gy)=\chi(g)d(x,y),
$ for all $g\in G,~ x,y\in M$.
In view of Proposition~\ref{prop:similarity-character}, the character
$\chi$ is uniquely determined by the action.

Let $E$ be a Banach space. The similarity character gives rise to the
natural scalar representation
$\rho_\chi\colon G\longrightarrow GL(E), ~~\rho_\chi(g)=\chi(g)I_E$.
It is therefore natural to consider mappings which are equivariant with
respect to the given similarity action of $G$ on $M$ and the
representation $\rho_\chi$ on $E$. The corresponding equivariance
condition is
$
f(gx)=\rho_\chi(g)f(x)=\chi(g)f(x),$ for all $ g\in G,~~x\in M$.
This motivates the following definition.

\begin{definition}\label{def:chi-equivariant}
A map
$f\in\Lip_0(M,E)$ is said to be \emph{$\chi$-equivariant} if
$f(gx)=\chi(g)f(x)$, for all $ g\in G,~ x\in M$.

We denote the space of all such mappings by
$\Lip_0^{G,\chi}(M,E)$.
In the scalar-valued case, we write simply $\Lip_0^{G,\chi}(M)
:=
\Lip_0^{G,\chi}(M,\mathbb R)$.
\end{definition}
Thus $\chi$-equivariance is precisely the usual equivariance condition
associated with the similarity action of $G$ on $M$ and the scalar
representation
$g\longmapsto \chi(g)I_E$
on $E$. In particular, if the action is isometric, then its similarity
character is the trivial character $\chi\equiv 1$, and the
$\chi$-equivariance condition reduces to
$f(gx)=f(x)$, for all $ g\in G,~ x\in M$.
In other words, $f$ is constant on the $G$-orbits of $M$. This is the
setting naturally associated with Lipschitz functions on the orbit
space $M/G$; see, for instance, \cite{PILFSIBGA}. Thus the case
$\chi\equiv 1$ may be regarded as the unweighted special case of the
character-equivariant framework introduced above.
\subsection{Lipschitz-free spaces}

Let $(M,d)$ be a metric space with distinguished point $0$ and $E$ be a Banach space.
We denote by $\operatorname{Lip}_0(M, Y)$ the Banach space of Lipschitz functions on $M$ vanishing at $0$, equipped
with the norm
$\Lip(f)
=
\sup_{ x\neq y}
\frac{\|f(x)-f(y)\|}{d(x,y)}.
$
For each $m\in M$, let $\delta_M(m)\in\operatorname{Lip}_0(M)^*$
be the evaluation functional given by
$ \delta_M(m)\left(f\right)=f(m)$ for all
$f\in\operatorname{Lip}_0(M)$.
Then the \emph{Lipschitz-free space} over $M$, also known as the
Arens--Eells space, is
\[
\mathcal{F}(M)
=
\overline{\operatorname{span}}
\{\delta_M(m):m\in M\}
\subseteq \operatorname{Lip}_0(M)^*,
\]
where the closure is taken in norm. The canonical map
$\delta_M\colon M\to\mathcal{F}(M)$ is an isometric embedding
with $\delta_M(0)=0$. Moreover, changing the distinguished point produces
a linearly isometric Lipschitz-free space.
The following universal property characterizes $\mathcal{F}(M)$
and expresses its role as a linearization of the pointed metric
space $M$.

\begin{theorem}\label{thm:free-universal-property} (\cite{LFBS,LA})
Let $(M,d)$ be a pointed metric space. For every Banach space
$E$ and every Lipschitz map $f\colon M\to E$ satisfying $f(0)=0$,
there exists a unique bounded linear operator
$\widehat{f}\colon\mathcal{F}(M)\to E$ such that
$\widehat{f}\circ\delta_M=f$.
Moreover,
$\|\widehat{f}\|=\operatorname{Lip}(f)$.
Together with the conditions that $\delta_M$ is an isometric
embedding, $\delta_M(0)=0$, and
$\overline{\operatorname{span}}~\delta_M(M)=\mathcal{F}(M)$,
this property determines the pair $(\mathcal{F}(M),\delta_M)$
uniquely up to a surjective linear isometry intertwining the
canonical embeddings.
\end{theorem}

Thus, the following diagram commutes:
\[
\begin{tikzcd}[column sep=large,row sep=large]
M \arrow[r,"f"] \arrow[dr,"\delta_M"']
  & E \\
  & \mathcal{F}(M) \arrow[u,"\widehat{f}"']
\end{tikzcd}
\]
For a proof and further background, see \cite[Section 2]{OTSOLFS}.

Applying Theorem~\ref{thm:free-universal-property} to scalar-valued
maps yields the canonical isometric identification
$\mathcal{F}(M)^*\cong\operatorname{Lip}_0(M)$.
In particular, $\mathcal{F}(M)$ is a canonical predual of
$\operatorname{Lip}_0(M)$. On norm-bounded subsets of
$\operatorname{Lip}_0(M)$, the associated weak$^*$ topology
coincides with the topology of pointwise convergence on $M$,
since $\operatorname{span}\delta_M(M)$ is dense in
$\mathcal{F}(M)$.

An alternative construction defines $\mathcal{F}(M)$ directly
through the Kantorovich--Rubinstein norm, without first introducing
$\operatorname{Lip}_0(M)$. This approach connects the theory of
Lipschitz-free spaces with optimal transport; see \cite{LA} for an
introduction.
\section{Character-equivariant Lipschitz mappings and their linearization}
\label{sec:character-equivariant}

In this section we develop the basic linearization theory associated with
character-equivariant Lipschitz mappings. We first realize these mappings
as the fixed points of a natural isometric representation on
$\Lip_0(M)$. We then pass to the corresponding action on the
Lipschitz-free space and identify the associated orbit-difference
subspace. This leads naturally to an equivariant Lipschitz-free space
and its universal property.

\subsection{The induced action on the Lipschitz space}
\label{subsec:action-Lip}

Let $(M,d,0)$ be a pointed metric space and let $G$ acts on $M$ by
similarities, fixing the distinguished point. Let
$\chi\colon G\to(0,\infty)$ denote the associated similarity character.
The action of $G$ on $M$ naturally induces an action on $\Lip_0(M)$.
For each $g\in G$, we define
\begin{equation}\label{eq:Ug}
U_g\colon \Lip_0(M)\longrightarrow \Lip_0(M), ~\mbox{by}~
(U_gf)(x)=\chi(g)f(g^{-1}x),
\end{equation}
for $f\in\Lip_0(M)$ and $x\in M$.

\begin{proposition}\label{prop:Ug-isometry}
For each $g\in G$, the operator $U_g$ defined by
\eqref{eq:Ug} is a surjective linear isometry on $\Lip_0(M)$.
Moreover,
$U_e=I_{\Lip_0(M)}~\text{and}~
U_gU_h=U_{gh}
~~\text{for all } g,h\in G$.
Consequently, the mapping
\[
U\colon G\longrightarrow
\operatorname{Isom}_{\mathrm{lin}}\bigl(\Lip_0(M)\bigr),
\qquad
g\longmapsto U_g,
\]
is a group homomorphism, and hence defines an isometric representation
of $G$ on $\Lip_0(M)$.
\end{proposition}

\begin{proof}
For any $g \in G$ immediately it follows that each $U_g$ is linear. Since the distinguished point is fixed by
the action,
$(U_gf)(0)
=
\chi(g)f(g^{-1}0)
=
0$.
Moreover, by the similarity property,
\[
d(g^{-1}x,g^{-1}y)
=
\chi(g^{-1})d(x,y)
=
\frac{1}{\chi(g)}d(x,y).
\]
Hence, for $x\neq y$,
\[
\begin{aligned}
\frac{|(U_gf)(x)-(U_gf)(y)|}{d(x,y)}
&=
\chi(g)
\frac{|f(g^{-1}x)-f(g^{-1}y)|}{d(x,y)}
=
\frac{|f(g^{-1}x)-f(g^{-1}y)|}
     {d(g^{-1}x,g^{-1}y)}.
\end{aligned}
\]
Taking the supremum over $x\neq y$ we have
$\Lip(U_gf)=\Lip(f)$.
Thus $U_g$ is an isometry.

Finally, for $g,h\in G$,
\[
\begin{aligned}
(U_gU_hf)(x)
&=
\chi(g)\chi(h)
f(h^{-1}g^{-1}x)
=
\chi(gh)f((gh)^{-1}x)
=
(U_{gh}f)(x).
\end{aligned}
\]
Therefore $U_gU_h=U_{gh}$. In particular,
$U_{g^{-1}}=U_g^{-1}$, and hence each $U_g$ is surjective.
\end{proof}

The character-equivariant Lipschitz mappings introduced in
Definition~\ref{def:chi-equivariant} admit a natural fixed-point description
with respect to this representation.
\begin{proposition}\label{prop:chi-fixed-points}
Let
$U=(U_g)_{g\in G}$
be the isometric representation of $G$ on $\Lip_0(M)$ given by
Proposition~\ref{prop:Ug-isometry}. Then the space of
$\chi$-equivariant Lipschitz mappings coincides with the fixed-point
space of $U$; that is,
\[
\Lip_0^{G,\chi}(M)
=
\Fix(U)
:=
\left\{
f\in\Lip_0(M):
U_gf=f \ \text{for every } g\in G
\right\}.
\]
\end{proposition}

\begin{proof}
Suppose $f\in\Lip_0^{G,\chi}(M)$. Then, for every
$g\in G$ and $x\in M$,
$$f(g^{-1}x)
=
\chi(g^{-1})f(x)
=
\frac{1}{\chi(g)}f(x).$$
Consequently,
$(U_gf)(x)
=
\chi(g)f(g^{-1}x)
=
f(x)$,
and hence $f\in\Fix(U)$.

Conversely, suppose that $U_gf=f$ for every $g\in G$. Then
$\chi(g)f(g^{-1}x)=f(x)$.
Now, replacing $x$ by $gx$ we get
$f(gx)=\chi(g)f(x)$,
and therefore $f\in\Lip_0^{G,\chi}(M)$.
\end{proof}

\begin{remark}\label{rem:trivial-character}
If $\chi\equiv1$, then the action of $G$ on $M$ is isometric and
$(U_gf)(x)=f(g^{-1}x)$.
In this case,
\[
\Lip_0^{G,\chi}(M)
=
\{f\in\Lip_0(M):f(gx)=f(x)
\text{ for all }g\in G,\ x\in M\}.
\]
Thus $f$ is constant on the $G$-orbits of $M$, and the preceding
construction reduces to the setting naturally associated with
Lipschitz functions on the orbit space $M/G$ (see \cite{PILFSIBGA}).
\end{remark}

\subsection{The induced action on the Lipschitz-free space}
\label{subsec:action-free}

We now pass from the fixed-point description on $\Lip_0(M)$ to its
predual counterpart. 

For $g\in G$, we define $V_g$ on the canonical vectors $\delta_M(x)$ for any $x \in M$ as
$V_g\delta_x
=
\frac{1}{\chi(g)}\delta_{gx}$.

\begin{proposition}\label{prop:Vg-action}
For every $g\in G$, the operator $V_g$ extends uniquely to a
surjective linear isometry
$V_g\colon\F(M)\longrightarrow\F(M)$. 
Moreover, $V_gV_h=V_{gh}$ for any $ g,h\in G,$ and
$V_g^*=U_{g^{-1}}$.
Equivalently, $U_g=V_{g^{-1}}^*$.
Thus $V=(V_g)_{g\in G}$ is an isometric representation of $G$ on
$\F(M)$ corresponding to the representation $(U_g)_{g\in G}$ on
$\Lip_0(M)$.
\end{proposition}

\begin{proof}
Let $g\in G$. Then for $x,y\in M$,
$$
\begin{aligned}
\|V_g\delta_x-V_g\delta_y\|=
\frac{1}{\chi(g)}
\|\delta_{gx}-\delta_{gy}\|
=
\frac{1}{\chi(g)}d(gx,gy)
=
d(x,y).
\end{aligned}
$$
Hence the map
$x\longmapsto\frac{1}{\chi(g)}\delta_{gx}$ 
is an isometric map from $M$ into $\F(M)$ which sends $0$ to $0$.
Thus, by the universal property of $\F(M)$, it extends uniquely to a
linear isometry $V_g\colon\F(M)\longrightarrow\F(M)$.

Again, for $g,h\in G$ and $x\in M$,
$$
\begin{aligned}
V_gV_h\delta_x
=
\frac{1}{\chi(h)}V_g\delta_{hx}
=
\frac{1}{\chi(g)\chi(h)}\delta_{ghx}
=
\frac{1}{\chi(gh)}\delta_{ghx}
=
V_{gh}\delta_x.
\end{aligned}
$$
It follows that
$V_gV_h=V_{gh}$.
In particular as, $V_{g^{-1}}=V_g^{-1}$, each $V_g$ is surjective.

Finally, for $f\in\Lip_0(M)$ and $x\in M$,
\[
\begin{aligned}
 (V_g^*f)(\delta_x)
=
(V_g\delta_x)(f)=
\frac{1}{\chi(g)}f(gx)
=
\chi(g^{-1})f(gx)
=
(U_{g^{-1}}f)(x).
\end{aligned}
\]
Hence
$V_g^*=U_{g^{-1}}$.
\end{proof}
The fixed-point description obtained above has a natural counterpart on
the Lipschitz-free space. Recall that, if $G$ acts linearly on a vector
space $Z$, the space of $G$-coinvariants is defined as
\[
Z_G
:=
Z\big/
\operatorname{span}\{gz-z:g\in G,\ z\in Z\};
\]
see, for instance,\cite[p.~131136]{UROI} and \cite[\S 7.6.4]{tomDieck}. In the Banach space setting, it is natural to
replace the linear span by its norm closure. Accordingly, for the
representation $V$ of $G$ on $\F(M)$ introduced above, we consider the
closed subspace
\begin{equation}\label{eq:Y-G-chi}
\mathcal Y_{G,\chi}
:=
\overline{\operatorname{span}}
\left\{
V_g\mu-\mu:
g\in G,\ \mu\in\F(M)
\right\}.
\end{equation}
We shall refer to $\mathcal Y_{G,\chi}$ as the
\emph{$(G,\chi)$-orbit-difference subspace} associated with the
representation $V$.

The preceding construction allows us to express the fixed-point space
of the representation $U$ in terms of the annihilator of the
orbit-difference subspace associated with $V$. More precisely, we have
the following.
\begin{theorem}\label{thm:annihilator-character}
Let $V=(V_g)_{g\in G}$ be the isometric representation of $G$ on
$\F(M)$ given by Proposition~\ref{prop:Vg-action}, and let
$\mathcal Y_{G,\chi}$ be its $(G,\chi)$-orbit-difference subspace
defined in \eqref{eq:Y-G-chi}. Then
\[
\mathcal Y_{G,\chi}^{\perp}
=
\Fix(U)
=
\Lip_0^{G,\chi}(M),
\]
where $U=(U_g)_{g\in G}$ is the isometric representation of $G$ on
$\Lip_0(M)$ introduced in Proposition~\ref{prop:Ug-isometry}.
Consequently,
${}^\perp\Lip_0^{G,\chi}(M)
=
\mathcal Y_{G,\chi}$.
\end{theorem}
\begin{proof}
Let $f\in\Lip_0(M)$. By the definition of
$\mathcal Y_{G,\chi}$, we have
$f\in\mathcal Y_{G,\chi}^{\perp}$ if and only if
\[
(V_g\mu-\mu)(f)=0
\qquad
\text{for every } g\in G \text{ and } \mu\in\F(M).
\]
In particular, for any $x \in M$ we have $(V_g\delta_x-\delta_x)(f)=0$. That is. $f \in \mathcal Y_{G,\chi}^{\perp}$.

Conversely, suppose that $f\in\Lip_0^{G,\chi}(M)$. Then $ f(gx)=\chi(g)f(x)$ for all $ g\in G,\ x\in M$, and therefore $(V_g\delta_x-\delta_x)(f) = \frac{1}{\chi(g)}f(gx)-f(x) = 0$ for every $g\in G$ and $x\in M$. Since $\operatorname{span}\{\delta_x:x\in M\}$ is dense in $\F(M)$ and $V_g-I_{\F(M)}$ is bounded, it follows that \[ (V_g\mu-\mu)(f)=0 \qquad \text{for every } g\in G \text{ and } \mu\in\F(M). \] Hence  $f\in\mathcal Y_{G,\chi}^{\perp}$. Consequently, $ \mathcal Y_{G,\chi}^{\perp} = \Lip_0^{G,\chi}(M)$.

Now we show that $f \in \mathcal Y_{G,\chi}^{\perp} \iff f \in \Fix(U)$.

Let $f \in \mathcal Y_{G,\chi}^{\perp}$. Thus,
$\frac{1}{\chi(g)}f(gx)-f(x)=0$,
and hence
$f(gx)=\chi(g)f(x)$ for all $g\in G,\ x\in M$.
Equivalently,
\[
f(g^{-1}x)
=
\chi(g^{-1})f(x)
=
\frac{1}{\chi(g)}f(x).
\]
Therefore, by the definition of $U_g$,
\[
(U_gf)(x)
=
\chi(g)f(g^{-1}x)
=
f(x) ~\mbox{for all}~g\in G,\ x\in M.
\]
Thus
$U_gf=f
~\mbox{for every }~ g\in G$,
and consequently $f\in\Fix(U)$.

Conversely, if $f\in\Fix(U)$, then
$\chi(g)f(g^{-1}x)=f(x)$ for all $g\in G,\ x\in M$.
Replacing $x$ by $gx$ we have
$f(gx)=\chi(g)f(x)$.
Hence
$$(V_g\delta_x-\delta_x)(f)
=
\frac{1}{\chi(g)}f(gx)-f(x)
=
0,$$
for every $g\in G$ and $x\in M$. Since
$\operatorname{span}\{\delta_x:x\in M\}$ is dense in $\F(M)$, it follows
that

$(V_g\mu-\mu)(f)=0
~\text{for every } g\in G,\ \mu\in\F(M)$.
Therefore
$f\in\mathcal Y_{G,\chi}^{\perp}$.
Consequently,
$\mathcal Y_{G,\chi}^{\perp}
=
\Fix(U)$.
By Proposition~\ref{prop:chi-fixed-points},
$\Fix(U)=\Lip_0^{G,\chi}(M),$
and hence
$\mathcal Y_{G,\chi}^{\perp}
=
\Fix(U)
=
\Lip_0^{G,\chi}(M)$.

Finally, since $\mathcal Y_{G,\chi}$ is a closed linear subspace of
$\F(M)$, the standard annihilator identity gives
${}^\perp\bigl(\mathcal Y_{G,\chi}^{\perp}\bigr)
=
\mathcal Y_{G,\chi}$.
Combining this with the preceding equality yields
${}^\perp\Lip_0^{G,\chi}(M)
=
\mathcal Y_{G,\chi}$,
as required.
\end{proof}

Motivated by Theorem~\ref{thm:annihilator-character}, we introduce
the corresponding quotient space.

\begin{definition}\label{def:equivariant-free}
The \emph{$(G,\chi)$-equivariant Lipschitz-free space} over $M$ is
defined by
\[
\F^{G,\chi}(M)
=
\F(M)/\mathcal Y_{G,\chi}.
\]
\end{definition}

The preceding theorem immediately gives the canonical dual
identification. From the standard duality for quotient spaces,
$\bigl(\F(M)/\mathcal Y_{G,\chi}\bigr)^*
\cong
\mathcal Y_{G,\chi}^{\perp}$.
The following corollary follows from
Theorem~\ref{thm:annihilator-character}.

\begin{corollary}\label{cor:equivariant-predual}
There is a canonical linear isometric identification
$\bigl(\F^{G,\chi}(M)\bigr)^*
\cong
\Lip_0^{G,\chi}(M)$.
\end{corollary}

\subsection{The universal property of the equivariant free space}
\label{subsec:universal-equivariant}

We conclude this section by showing that $\F^{G,\chi}(M)$ plays the
same linearizing role for character-equivariant Lipschitz mappings as
$\F(M)$ does for arbitrary Lipschitz mappings.

Let $E$ be a Banach space and let
$q_{G,\chi}\colon
\F(M)\longrightarrow\F^{G,\chi}(M)$
denote the canonical quotient map.

\begin{theorem}\label{thm:universal-character}
For every Banach space $E$, the correspondence
\[
f\longmapsto\widetilde f,
\qquad
\widetilde f\bigl(q_{G,\chi}(\delta_x)\bigr)=f(x),
\]
defines a canonical linear isometric isomorphism
$\Lip_0^{G,\chi}(M,E)
\cong
\mathcal L\bigl(\F^{G,\chi}(M),E\bigr)$.
\end{theorem}

\begin{proof}
Let $f\in\Lip_0^{G,\chi}(M,E)$. By the universal property of
$\F(M)$, there exists a unique operator $\widehat f\colon\F(M)\longrightarrow E$
such that
$\widehat f(\delta_x)=f(x),$ for all $ x\in M$,
and $\|\widehat f\|=\Lip(f)$.
For $g\in G$ and $x\in M$,
\[
\begin{aligned}
\widehat f
\bigl(\delta_{gx}-\chi(g)\delta_x\bigr)
&=
f(gx)-\chi(g)f(x)=0.
\end{aligned}
\]
Hence
$\mathcal Y_{G,\chi}\subseteq\ker\widehat f$,
and therefore $\widehat f$ factors uniquely through the quotient.
Thus there exists a unique

$\widetilde f\colon\F^{G,\chi}(M)\longrightarrow E$
such that
$\widehat f=\widetilde f\circ q_{G,\chi}$.
Moreover, $\|\widetilde f\|\leq\Lip(f)$.

Conversely, let
$T\in\mathcal L(\F^{G,\chi}(M),E)$
and define
$f_T(x)
=
T\bigl(q_{G,\chi}(\delta_x)\bigr),$ for all $ x\in M$.
Since $q_{G,\chi}$ is contractive,
\[
\begin{aligned}
\|f_T(x)-f_T(y)\|
&\leq
\|T\|
\|q_{G,\chi}(\delta_x-\delta_y)\|
\leq
\|T\|d(x,y).
\end{aligned}
\]
Thus $f_T\in\Lip_0(M,E)$ and $\Lip(f_T)\leq\|T\|$.
Furthermore,
$q_{G,\chi}
\bigl(\delta_{gx}-\chi(g)\delta_x\bigr)
=0$,
so $f_T(gx)=\chi(g)f_T(x)$.
Hence $f_T\in\Lip_0^{G,\chi}(M,E)$.

The above two constructions are inverse to each other. Consequently, $\|\widetilde f\|=\Lip(f)$,
and the correspondence is a linear isometric isomorphism.
\end{proof}
%================================================
\subsection{Amenable averaging and bidual realizations}
\label{subsec:amenable-bidual}
%================================================

The orbit-difference construction is closely related to
the averaging method of C\'uth and Doucha \cite{PILFSIBGA}.
Their Lemma~3.2 associates with an isometric group action
on a Banach space a contractive averaging operator taking
values in its bidual. We record the consequence of this
construction for the Banach spaces of coinvariants
considered here.

For the next proposition we consider $G$ be a group that is amenable when endowed with the discrete topology.
\begin{proposition}\label{prop:amenable-coinvariants}
Let $X$ be a real Banach space, let $G$ be an amenable
discrete group, and let
$V\colon G\longrightarrow\operatorname{Isom}_{\mathrm{lin}}(X)$
be a representation by surjective linear isometries.
Set\\
$\mathcal D_V
=
\overline{\operatorname{span}}
\{V_g u-u:g\in G,\ u\in X\}$, and let $q\colon X\to X/\mathcal D_V$ be the quotient map.
Then there exists a linear isometric embedding
$J_V\colon X/\mathcal D_V\longrightarrow X^{**}$.
Moreover, the subspace\\
$\mathcal D_V^\perp
=
\{f\in X^*:V_g^*f=f\text{ for every }g\in G\}$
is the range of a contractive projection on $X^*$.
\end{proposition}

\begin{proof}
Let us choose a bi-invariant mean $m$ on $\ell_\infty(G)$.
Following \cite[Lemma~3.2]{PILFSIBGA}, we consider
$R_V\colon X\to X^{**}$ as
$(R_Vu)(f)
=
m_g\bigl[g \mapsto f(V_{g^{-1}}u)\bigr],
~\mbox{ for any}~ u\in X,\ f\in X^*$.
Then $\|R_V\|\leq1$, and invariance of the mean gives $R_VV_h=R_V $ for any $h\in G$.
Hence $R_V$ vanishes on $\mathcal D_V$ and factors
through a contraction
\[
J_V\colon X/\mathcal D_V\longrightarrow X^{**},
\qquad J_Vq=R_V.
\]

Let $f\in\mathcal D_V^\perp$, then
$f(V_{g^{-1}}u)=f(u)$ for every $g\in G, u \in X$, and therefore $(R_Vu)(f)=f(u)$.
Using the canonical identification
$(X/\mathcal D_V)^*=\mathcal D_V^\perp$, we obtain
\[
\|q(u)\|
=
\sup_{\substack{f\in\mathcal D_V^\perp\\\|f\|\leq1}}
|f(u)|\leq \|R_Vu\|
=\|J_Vq(u)\|
\leq\|q(u)\|.
\]
Thus $J_V$ is an isometry. In particular,
$\ker R_V=\mathcal D_V$.

Finally, we define $P_V\colon X^*\to X^*$ by $(P_Vf)(u)= (R_Vu)(f)$.
The identity $R_VV_h=R_V$ implies that
$P_Vf\in\mathcal D_V^\perp$, while the preceding
calculation shows that $P_Vf=f$ whenever
$f\in\mathcal D_V^\perp$. Thus $P_V$ is a contractive
projection onto $\mathcal D_V^\perp$.
\end{proof}
\begin{remark}
The discreteness assumption allows us to average arbitrary bounded coefficient functions $g\mapsto f(V_{g^{-1}}u)$ without imposing continuity conditions on the representation. Thus amenability is understood here for the underlying discrete group. A left-invariant mean on $\ell_\infty(G)$ suffices for the proof.
\end{remark}

Applying this proposition to the normalized representation
on $\F(M)$ gives the following consequence.

\begin{corollary}\label{cor:equivariant-bidual}
Suppose that an amenable discrete group $G$ acts on a
pointed metric space $M$ by similarities fixing the
distinguished point, with associated character $\chi$.
Then
$\F^{G,\chi}(M)$
admits a linear isometric embedding into $\F(M)^{**}$.
Moreover, $\Lip_0^{G,\chi}(M)$ is the range of a
contractive projection on $\Lip_0(M)$, given by $(P_{G,\chi}f)(x)
=
m_g\left[g \mapsto \chi(g)f(g^{-1}x)\right]$.
If $\Lip_0^{G,\chi}(M)\neq\{0\}$, this projection
has norm one.
\end{corollary}

\begin{proof}
Applying Proposition~\ref{prop:amenable-coinvariants} with
$X=\F(M)$ and
$V_g\delta_x=\frac{1}{\chi(g)}\delta_{gx}$.
The orbit-difference subspace is $\mathcal Y_{G,\chi}$,
and its annihilator is $\Lip_0^{G,\chi}(M)$.
Furthermore,
$V_{g^{-1}}\delta_x=\chi(g)\delta_{g^{-1}x}$,
which gives the stated projection formula.
\end{proof}

\begin{remark}
When $\chi\equiv1$, the quotient
$\F^{G,1}(M)$ is canonically isometric to $\F(M/G)$,
where $M/G$ denotes the metric space of orbit closures
used in \cite{PILFSIBGA}. This is the quotient identification
of \cite[Lemma~4.3]{PILFSIBGA}.

Thus Corollary~\ref{cor:equivariant-bidual} gives the
point-preserving isometric case of the bidual realization
in \cite[Theorem~4.5]{PILFSIBGA}, and extends the same
construction to normalized similarity representations.
\end{remark}

%================================================
\section{Applications}
\label{sec:applications}
%================================================

We now present several applications of the fixed-point and
orbit-difference framework developed in the preceding section. We first
show that positive homogeneous Lipschitz mappings arise from the
character-equivariant construction associated with the dilation action
of the multiplicative group $(0,\infty)$. We then turn to recentered
translation actions and show that their fixed points are precisely the
bounded linear mappings. Finally, we extend the same point of view to
two-Lipschitz mappings, where the corresponding fixed-point space is
the space of bounded bilinear mappings.

%================================================
\subsection{Positive homogeneous Lipschitz mappings} (\cite{AM2})
\label{subsec:positive-homogeneous}
%================================================

We first apply the character-equivariant framework to positive
homogeneous Lipschitz mappings. Let $X$ be a Banach space and consider
the multiplicative group
$G=(0,\infty)$
acting on $X$ by dilations, $r\cdot x=rx$.

Let $\chi(r)=r$.
Then $\|rx-ry\|
=
r\|x-y\|
=
\chi(r)\|x-y\|$,
so that this is an action by similarities with associated similarity
character $\chi$ and hence the results of the preceding section apply. We recall that $\Lip_0^{ph}(X) = \left\{ f\in\Lip_0(X): f(rx)=rf(x) \text{ for every } r\geq0,\ x\in X \right\}$ (see \cite{AM2}). With respect to the above action and character, the condition $f(rx)=rf(x)$ is precisely the $(G,\chi)$-equivariance condition. Consequently, \begin{equation}\label{eq:ph-character-equivariant} \Lip_0^{ph}(X) = \Lip_0^{G,\chi}(X). \end{equation} The representation of $G$ on $\Lip_0(X)$ introduced in \eqref{eq:Ug} takes, in the present setting, the form \begin{equation}\label{eq:dilation-action-Lip} (U_rf)(x) = r f\left(\frac{x}{r}\right), \qquad r>0,\ x\in X. \end{equation} Therefore, by Proposition~\ref{prop:chi-fixed-points} and \eqref{eq:ph-character-equivariant}, we immediately obtain \[ \Lip_0^{ph}(X) = \Fix(U) = \left\{ f\in\Lip_0(X): U_rf=f \text{ for every } r>0 \right\}. \] Thus positive homogeneous Lipschitz mappings are precisely the fixed points of the normalized dilation representation on $\Lip_0(X)$. We next consider the corresponding representation on the Lipschitz-free space. By Proposition~\ref{prop:Vg-action}, we have $ V_g\delta_x = \frac{1}{\chi(g)}\delta_{gx}$. Hence, for the dilation action, \begin{equation}\label{eq:dilation-action-free} V_r\delta_x = \frac1r\delta_{rx}, \qquad r>0,\ x\in X. \end{equation} It follows that the associated $(G,\chi)$-orbit-difference subspace is \[ \begin{aligned} \mathcal Y_{G,\chi} = \overline{\operatorname{span}} \left\{ V_r\delta_x-\delta_x: r>0,\ x\in X \right\} = \overline{\operatorname{span}} \left\{ \frac1r\delta_{rx}-\delta_x: r>0,\ x\in X \right\}. \end{aligned} \] Since multiplication of the generators by the nonzero scalar $r$ does not change their closed linear span, we obtain \begin{equation}\label{eq:Y-ph} \mathcal Y_{G,\chi} = \overline{\operatorname{span}} \left\{ \delta_{rx}-r\delta_x: r>0,\ x\in X \right\} = \overline{\operatorname{span}} \left\{ r\delta_x-\delta_{rx}: r>0,\ x\in X \right\}. \end{equation} Combining \eqref{eq:ph-character-equivariant}, \eqref{eq:Y-ph}, and Theorem~\ref{thm:annihilator-character}, we obtain \begin{equation}\label{eq:ph-preannihilator} {}^\perp\Lip_0^{ph}(X) = \overline{\operatorname{span}} \left\{ r\delta_x-\delta_{rx}: r>0,\ x\in X \right\}. \end{equation} Consequently, the corresponding equivariant free space is \[ \F^{G,\chi}(X) = \bigslant{\F(X)} {\overline{\operatorname{span}} \left\{ r\delta_x-\delta_{rx}: r>0,\ x\in X \right\}}. \] The quotient on the right-hand side is precisely the canonical predual $\F^{ph}(X)$ of $\Lip_0^{ph}(X)$ as described in \cite[Proposition ~3.5]{AM2}. Hence \begin{equation}\label{eq:equivariant-free-ph} \F^{G,\chi}(X) \cong \F^{ph}(X) \end{equation} linearly isometrically. Thus the canonical relations $r\delta_x-\delta_{rx}$ for all $ r>0,\ x\in X$, which determine $\F^{ph}(X)$, have a natural group-theoretic interpretation: they are precisely the orbit differences arising from the dilation representation of the multiplicative group $(0,\infty)$ on $\F(X)$. In particular, the positive-homogeneous free space may be viewed as the corresponding Banach space of coinvariants. Finally, the universal property of the equivariant free space specializes to the positive homogeneous setting. For every Banach space $E$, one has the canonical linear isometric identification $ \Lip_0^{ph}(X,E) \cong \mathcal L(\F^{ph}(X),E)$, where the operator associated with $f\in\Lip_0^{ph}(X,E)$ is determined by  $\widetilde f(q_{G,\chi}(\delta_x))=f(x)$, for all $x\in X$. 
% Continue here with:
% - the representation U_r;
% - Lip_0^{ph}(X)=Fix(U);
% - the induced action on F(X);
% - the orbit-difference description;
% - identification with F^{ph}(X);
% - the universal property.
\subsubsection{Amenable averaging and positive homogeneity}

In \cite[Proposition~4.2.5]{MandalThesis}, Mandal established
that $\Lip_0^{ph}(X)$ is complemented in $\Lip_0(X)$ by
constructing the radial extension projection
\[
(Qf)(x)
=
\begin{cases}
\|x\|f\left(\dfrac{x}{\|x\|}\right), & x\neq 0,\\[6pt]
0, & x=0,
\end{cases}
\]
with $\|Q\|\leq 3$. This estimate does not establish
$1$-complementability. We now strengthen the complementability
conclusion by constructing a norm-one projection through
amenable averaging.

Since the multiplicative group $(0,\infty)$ is abelian, it
admits an invariant mean
$m\colon\ell_\infty((0,\infty))\to\mathbb{R}$.
For $f\in\Lip_0(X)$ and $x\in X$, we define
\begin{equation}\label{eq:Pph}
(P_{ph}f)(x)
=
m_r
\left[r \mapsto
r f\left(\frac{x}{r}\right)
\right],
\end{equation}
where $m_r$ means applying the invariant mean $m$ to a bounded function of the variable $r$, explicitly to the map $ r \mapsto r f\left(\frac{x}{r}\right)$.
The expression is well defined, since for any fix $x \in X$ and for any $r>0$ we have  
$\left|
r f\left(\frac{x}{r}\right)
\right|
\leq
\Lip(f)\|x\| $ .

\begin{theorem}\label{thm:ph-projection}
There exists a norm-one projection
\[
P_{ph}\colon
\Lip_0(X)\longrightarrow\Lip_0^{ph}(X).
\]
Consequently, $\Lip_0^{ph}(X)$ is $1$-complemented in
$\Lip_0(X)$.
\end{theorem}

\begin{proof}
Let $f\in\Lip_0(X)$. For $x,y\in X$,
\[
\begin{aligned}
|(P_{ph}f)(x)-(P_{ph}f)(y)|
&\leq
m_r
\left[ r \mapsto
r\left|
f\left(\frac{x}{r}\right)
-
f\left(\frac{y}{r}\right)
\right|
\right]
\leq
\Lip(f)\|x-y\|.
\end{aligned}
\]
Hence
$\Lip(P_{ph}f)\leq\Lip(f)$.

We next show that $P_{ph}f$ is positive homogeneous. Let $s>0$.
Then
$(P_{ph}f)(sx)
=
m_r
\left[
r f\left(\frac{sx}{r}\right)
\right]$.

Put
$H_x(t)=t f\left(\frac{x}{t}\right)$. Thus
we have
$r f\left(\frac{sx}{r}\right)
=
sH_x\left(\frac{r}{s}\right)$.
By invariance of the mean under multiplication by $s^{-1}$, we get
$m_r
H_x\left(\frac{r}{s}\right)
=
m_r H_x(r)$.
It follows that
$(P_{ph}f)(sx)
=
s(P_{ph}f)(x)$, and hence
$P_{ph}f\in\Lip_0^{ph}(X)$.

If $f\in\Lip_0^{ph}(X)$, then
for any $r>0~~
r f\left(\frac{x}{r}\right)=f(x)$. That is.
$P_{ph}f=f$.
Thus $P_{ph}$ is a projection onto $\Lip_0^{ph}(X)$.
Finally,
$\|P_{ph}\|\leq1$.
Since $P_{ph}$ restricts to the identity on the nonzero space
$\Lip_0^{ph}(X)$, we have
$\|P_{ph}\|=1$. Hence the proof follows.
\end{proof}

\begin{remark}\label{rem:comparison-ph-projections}
\begin{enumerate}
    \item The averaging projection $P_{ph}$ is distinct from the radial
projection $Q$ above. Indeed, for $s>0$, we define
$(U_s f)(x)=s f\left(\frac{x}{s}\right)$.
Invariance of the mean gives
\[
(P_{ph}U_s f)(x)
=
m_r\left[rs f\left(\frac{x}{rs}\right)\right]
=
(P_{ph}f)(x),
\]
and hence $P_{ph}U_s=P_{ph}$ for every $s>0$.

The radial projection does not satisfy this identity.
For example, let $f(x)=\min\{\|x\|,1\}$. Then we have
$(Qf)(x)=\|x\|,~\mbox{and}~
(QU_{1/2}f)(x)=\frac12\|x\|$.
Since $X\neq\{0\}$, these functions are different. Therefore
$P_{ph}\neq Q$, independently of the choice of invariant mean.
Theorem~\ref{thm:ph-projection} thus establishes
$1$-complementability by a different construction; it does not
assert that the radial projection has norm one.

\item The averaging argument does not establish weak$^*$ continuity
of $P_{ph}$. Therefore, it does not by itself provide a
norm-one projection onto an isometric copy of $\F^{ph}(X)$
inside $\F(X)$. Here $\F^{ph}(X)$ is naturally realized as
a quotient of $\F(X)$, rather than as a specified subspace.
Obtaining a compatible complemented realization at the
predual level we require an additional argument.
\end{enumerate}

\end{remark}

% Continue with your invariant-mean construction P_{ph} and
% Theorem~\ref{thm:ph-projection}.

%================================================
%================================================
\subsection{Recentered translations and Lipschitz mappings}
\label{subsec:linear-applications}
%================================================

We next apply the fixed-point and orbit-difference approach to bounded
linear mappings. In contrast with the dilation action considered in the
preceding subsection, the additive group $(X,+)$ acts on $X$ by
translations, which do not preserve the distinguished point $0$.
Consequently, the character-equivariant construction developed above
does not apply directly. This difficulty is overcome by recentering
translations at the origin.

\subsubsection{The translation action and its fixed points}
Let $X$ and $Y$ be real Banach spaces. For $a\in X$ and
$f\in\Lip_0(X,Y)$, we define
\begin{equation}\label{eq:rho-translation}
(\rho_a f)(x)
=
f(x+a)-f(a), ~\mbox{for any}~ x\in X.
\end{equation}
Clearly, $\rho_a f$ vanishes at the origin. Moreover,
\[
\|(\rho_a f)(x)-(\rho_a f)(x')\|
=
\|f(x+a)-f(x'+a)\|
\leq
\Lip(f)\|x-x'\|,
\]
and hence $\rho_a f\in\Lip_0(X,Y)$. The following proposition shows
that these recentered translations form an isometric representation
of the additive group $X$.

\begin{proposition}\label{prop:rho-translation-action}
For every $a\in X$, the operator
$\rho_a\colon\Lip_0(X,Y)\longrightarrow\Lip_0(X,Y)$
is a surjective linear isometry. Moreover,
$\rho_a\rho_b=\rho_{a+b}$, for any $ a,b\in X$.
Consequently, $(\rho_a)_{a\in X}$ is an isometric representation of
the additive group $(X,+)$ on $\Lip_0(X,Y)$.
\end{proposition}

\begin{proof}
Linearity of $\rho_a$ immediately follows for any $a \in X$. As observed above,
$\rho_a f\in\Lip_0(X,Y)$ with
$\Lip(\rho_a f)\leq\Lip(f)$.

Furthermore, for any $a,b,x\in X$,
\[
\begin{aligned}
(\rho_a\rho_bf)(x)
&=
(\rho_bf)(x+a)-(\rho_bf)(a)=
f(x+a+b)-f(b)-f(a+b)+f(b)\\
&=
f(x+a+b)-f(a+b)=
(\rho_{a+b}f)(x).
\end{aligned}
\]
Thus $\rho_a\rho_b=\rho_{a+b}$,
in particular,
$\rho_{-a}=\rho_a^{-1}$.
Now, applying the preceding Lipschitz estimate to $\rho_{-a}$ we get
\[
\Lip(f)
=
\Lip(\rho_{-a}\rho_af)
\leq
\Lip(\rho_af)
\leq
\Lip(f),
\]
and therefore $\Lip(\rho_af)=\Lip(f)$.
Hence, each $\rho_a$ is a surjective linear isometry.
\end{proof}
The fixed-point space of this representation has a simple and useful
description.

\begin{theorem}\label{thm:linear-fixed-points}
Let $X$ and $Y$ be real Banach spaces, and let
\[
\rho\colon (X,+)\longrightarrow
\operatorname{Isom}_{\mathrm{lin}}\bigl(\Lip_0(X,Y)\bigr),
\qquad
a\longmapsto \rho_a,
\]
be the isometric representation defined as in \ref{eq:rho-translation}.
Then the fixed-point space of $\rho$ coincides with the space of
bounded linear operators from $X$ into $Y$. More precisely,
\[
\Fix(\rho)
:=
\bigcap_{a\in X}\ker\bigl(\rho_a-I_{\Lip_0(X,Y)}\bigr)
=
\mathcal L(X,Y).
\]
Equivalently,
$f\in\mathcal L(X,Y)
\quad\Longleftrightarrow\quad
\rho_a f=f
\quad\text{for every }a\in X$.
\end{theorem}

\begin{proof}
If $T\in\mathcal L(X,Y)$, then, for every $a,x\in X$,
\[
(\rho_aT)(x)
=
T(x+a)-T(a)
=
T(x),
\]
and hence $T\in\Fix(\rho)$.

Conversely, suppose that $f\in\Fix(\rho)$. Then
$f(x+a)-f(a)=f(x)$, for any $ a,x\in X$,
or equivalently,
$f(x+a)=f(x)+f(a)$.
Thus $f$ is additive. Since $f$ is Lipschitz, it is continuous, and
hence the usual continuity argument for additive mappings yields
$f(\lambda x)=\lambda f(x)$, for any $\lambda\in\mathbb R,\ x\in X$.
Therefore, $f\in\mathcal L(X,Y)$.
\end{proof}

Thus bounded linear mappings are precisely the fixed points of the
recentered translation representation. This provides the analogue, for
linearity, of the fixed-point description of positive homogeneous
Lipschitz mappings obtained in the preceding subsection.

%================================================
\subsubsection{The induced action and the translation quotient}
%================================================

We now describe the corresponding action on $\F(X)$. For $a\in X$,
we consider the map
$X\longrightarrow\F(X),
$ such that $
x\longmapsto\delta_{x+a}-\delta_a$.
It sends $0$ to $0$ and satisfies
$\|(\delta_{x+a}-\delta_a)
-(\delta_{x'+a}-\delta_a)\|
=
\|x-x'\|$.
Hence, by the universal property of $\F(X)$, there exists a unique
linear isometry
$S_a\colon\F(X)\longrightarrow\F(X)$
such that
\begin{equation}\label{eq:Sa-definition}
S_a\delta_x
=
\delta_{x+a}-\delta_a.
\end{equation}
\begin{proposition}\label{prop:Sa-action}
Let $X$ be a real Banach space. For each $a\in X$, let
$S_a\colon \F(X)\longrightarrow \F(X)$
be the linear isometry determined by
$S_a\delta_x=\delta_{x+a}-\delta_a$ for any $ x\in X$.
Then each $S_a$ is surjective, and
$S_0=I_{\F(X)}
\qquad\text{and}\qquad
S_aS_b=S_{a+b}~\text{for all }a,b\in X$.
Consequently, the mapping
\[
S\colon (X,+)\longrightarrow
\operatorname{Isom}_{\mathrm{lin}}\bigl(\F(X)\bigr),
\qquad
a\longmapsto S_a,
\]
is a group homomorphism, and hence defines an isometric representation
of $(X,+)$ on $\F(X)$.

Moreover, under the canonical isometric identification
$\F(X)^*=\Lip_0(X)$
we have for every $a\in X$,
$S_a^*=\rho_a$.
More generally, for
$f\in\Lip_0(X,Y)$ and any $a \in X~$, $\widehat{\rho_af}= \widehat{f} \circ S_a$.
\end{proposition}

\begin{proof}
For any $a,b,x\in X$, it follows that
$$S_aS_b\delta_x
=
S_a(\delta_{x+b}-\delta_b)
=
\delta_{x+a+b}-\delta_a
-\delta_{a+b}+\delta_a
=
\delta_{x+a+b}-\delta_{a+b}
=
S_{a+b}\delta_x.$$

Thus
$S_aS_b=S_{a+b}$.
In particular, $S_{-a}=S_a^{-1}$,
so each $S_a$ is bijective.

Now, for any $f\in\Lip_0(X)$, 
$ (S_a^*f)(x)= (S_a\delta_x)(f)=
f(x+a)-f(a)=
(\rho_af)(x)$.
Hence $S_a^*=\rho_a$.

For the vector-valued assertion,
\[
(\widehat{f}\circ S_a)(\delta_x)=
\widehat{f}(\delta_{x+a}-\delta_a)
=
f(x+a)-f(a)
=
(\rho_af)(x)=
\widehat{\rho_af}(\delta_x).
\]
Since the linear span of $\{\delta_x:x\in X\}$ is dense in
$\F(X)$, the proof follows.
\end{proof}

%================================================
%================================================
\subsubsection{The translation orbit-difference space}
\label{subsubsec:translation-orbit-difference}
%================================================

We now apply the orbit-difference construction developed in the
preceding section to the recentered translation representation $S$ as defined in the Proposition \ref{prop:Sa-action}.
Accordingly, in analogy with the $(G,\chi)$-orbit-difference subspace
introduced in \eqref{eq:Y-G-chi}, we define
\begin{equation}\label{eq:translation-orbit-space}
\mathcal Y_X
:=
\overline{\operatorname{span}}
\left\{
S_a\mu-\mu:
a\in X,\ \mu\in\F(X)
\right\}.
\end{equation}
Since $\operatorname{span}\{\delta_x:x\in X\}$ is dense in $\F(X)$
and each $S_a-I_{\F(X)}$ is continuous, it follows that
\begin{equation}\label{eq:translation-generators}
\mathcal Y_X
=
\overline{\operatorname{span}}
\left\{
\delta_{x+a}-\delta_x-\delta_a:
x,a\in X
\right\}.
\end{equation}

As in the general fixed-point/orbit-difference correspondence, the
annihilator of $\mathcal Y_X$ is precisely the fixed-point space of
the dual representation. Indeed, Proposition~\ref{prop:Sa-action}
gives $S_a^*=\rho_a$, for all $a\in X$,
under the canonical identification $\F(X)^*=\Lip_0(X)$. Hence
\[
\mathcal Y_X^\perp
=
\left\{
f\in\Lip_0(X):
S_a^*f=f\text{ for every }a\in X
\right\}=
\Fix(\rho).
\]
By Theorem~\ref{thm:linear-fixed-points}, applied with scalar range,
$\Fix(\rho)=X^*$. Therefore
\begin{equation}\label{eq:translation-annihilator}
\mathcal Y_X^\perp=X^*.
\end{equation}

This abstract description of $\mathcal Y_X$ admits a concrete
identification in terms of the barycentre map
$\beta_X : \F(X)\xrightarrow[contraction]{onto} X$. Then $\beta_X^*$ is an isometry. Hence $\operatorname{ran}(\beta_X^*)=X^*$.

\begin{remark}\label{rem:kernel-orbit}
\begin{enumerate}
\item Since $(\ker\beta_X)^\perp
=
\operatorname{ran}(\beta_X^*)$, the bipolar theorem yields
\[
\mathcal Y_X
=
{}^\perp(\mathcal Y_X^\perp)
=
{}^\perp\bigl((\ker\beta_X)^\perp\bigr)
=
\ker(\beta_X).
\]

\item 
Thus we identify the kernel of the
barycentre map with the orbit-difference subspace associated with the
recentered translation representation. In particular,
\[
\ker(\beta_X)
=
\overline{\operatorname{span}}
\left\{
\delta_{x+a}-\delta_x-\delta_a:
x,a\in X
\right\}.
\]
Moreover, the generators of $\ker(\beta_X)$ are precisely the defects of
additivity of the canonical embedding
$\delta_X\colon X\longrightarrow\F(X)~ \mbox{given by}~
x\longmapsto\delta_x$,
since $\delta_{x+a}-\delta_x-\delta_a$
measures the failure of $\delta_X$ to satisfy
$\delta_{x+a}=\delta_x+\delta_a$.
Accordingly, the decomposition into the orbit-difference space and
its quotient separates two aspects of the canonical embedding:
$\ker(\beta_X)~~\text{records its failure to be additive part,}$
whereas $\F(X)/\ker(\beta_X)\cong X$
recovers the underlying linear structure of $X$.
\end{enumerate}
\end{remark}
%================================================
The preceding identifications yield the following equivalent
characterizations of linearity.

\begin{corollary}\label{thm:three-linear-characterizations}
Let $X$ and $Y$ be real Banach spaces. For
$f\in\Lip_0(X,Y)$, the following assertions are equivalent:
\begin{enumerate}
\item[(i)] $f\in\mathcal L(X,Y)$;
\item[(ii)] $\rho_af=f$ for every $a\in X$;
\item[(iii)] $\widehat f\circ S_a=\widehat f$ for every $a\in X$;
\item[(iv)] $\widehat f|_{\ker(\beta_X)}=0$.
\end{enumerate}
\end{corollary}

\begin{proof}
The equivalence of (i) and (ii) follows from
Theorem~\ref{thm:linear-fixed-points}. By The Proposition
\eqref{prop:Sa-action}, it follows
$\widehat{\rho_af}=\widehat{f}\circ S_a$,
and the uniqueness of linearization therefore gives the equivalence
of $(ii)$ and $(iii)$.

Suppose that $(iii)$ holds. Then
$\widehat{f}(S_a\mu-\mu)=0$ for any $a\in X,\ \mu\in\F(X)$.
Now from Remark \ref{rem:kernel-orbit} we have
$\ker(\beta_X)
=
\mathcal Y_X
=
\overline{\operatorname{span}}
\{S_a\mu-\mu:a\in X,\ \mu\in\F(X)\}$ and hence
we obtain $(iv)$.

Conversely, let $\widehat{f}$ vanishes on $\ker(\beta_X)$, then
as $S_a\mu-\mu\in\ker(\beta_X)$;
for every $a\in X$ and $\mu\in\F(X)$ we have
$\widehat{f} \circ S_a=\widehat{f}$, which gives $(iii)$.
\end{proof}

%================================================
In particular, bounded linear mappings on $X$ correspond
isometrically to bounded linear operators on the translation
quotient.

\begin{corollary}\label{cor:linear-universal-property}
Let $X$ and $Y$ be real Banach spaces, and let
$q_X\colon\F(X)\to\F(X)/\mathcal Y_X$ be the quotient map.
For every $T\in\mathcal L(X,Y)$, there exists a unique
bounded linear operator
\[
\widetilde T\colon
\bigslant{\F(X)}{\mathcal Y_X}\longrightarrow Y
\]
such that
$\widetilde T\circ q_X\circ\delta_X=T$.
Moreover, $\|\widetilde T\|=\|T\|$, and the correspondence
$T\mapsto\widetilde T$ is a surjective linear isometry.
Consequently,
$\mathcal L(X,Y)
\cong
\mathcal L\left(
\bigslant{\F(X)}{\mathcal Y_X},Y
\right)
=
\mathcal L\left(
\bigslant{\F(X)}{\ker(\beta_X)},Y
\right)$.
\end{corollary}

%================================================
%================================================
\subsubsection{Amenable averaging and the projection onto linear mappings}
%================================================

The fixed-point description obtained above also provides a
group-theoretic interpretation of the projection onto
$\mathcal L(X,Y)$ constructed in \cite{AM3}. Indeed, the additive group
$(X,+)$ is abelian and hence amenable. If $Y$ is a dual Banach space,
the invariant-mean construction of \cite{AM3} associates with
$f\in\Lip_0(X,Y)$ the mapping
\begin{equation}\label{eq:translation-projection}
(Pf)(z)
=
\mathfrak M_x
\bigl[x \mapsto (f(x+z)-f(x))\bigr],
\qquad z\in X,
\end{equation}
where $\mathfrak M$ is a norm-one translation-invariant
$Y$-valued mean. For details we refer \cite{GNFA}.

In terms of the recentered translation representation introduced in \ref{eq:rho-translation}, we have
\begin{equation}\label{eq:projection-orbit-average}
(Pf)(z)
=
\mathfrak M_x\bigl[x \mapsto(\rho_xf)(z)\bigr].
\end{equation}
Thus the projection constructed in \cite{AM3} is precisely the
invariant-mean averaging of the orbit of $f$ under the representation
$\rho$.

Recall from \cite{AM3} that this construction yields the following.

\begin{theorem}\label{thm:projection-linear-fixed}
Let $X$ be a real Banach space and let $Y$ be a dual Banach space.
Then $\mathcal L(X,Y)$ is $1$-complemented in $\Lip_0(X,Y)$. More
precisely, the operator $P$ defined by
\eqref{eq:translation-projection} is a norm-one projection
$P\colon\Lip_0(X,Y)\longrightarrow\mathcal L(X,Y)$.
\end{theorem}

In view of Theorem~\ref{thm:linear-fixed-points}, the range of this
projection may now be written as
$\operatorname{ran}(P)
=
\mathcal L(X,Y)
=
\Fix(\rho)$.
Consequently, the projection of \cite{AM3} admits the following
interpretation within the present framework:
$P
=
\text{invariant-mean averaging onto the fixed-point space of }\rho$.
Thus the invariant-mean construction from \cite{AM3} is an instance of
the general principle that averaging an amenable group action produces
a projection onto its fixed-point space.

%================================================
\subsubsection{The quotient by the linear mappings}
%================================================

The orbit-difference description obtained above also gives a
group-theoretic interpretation of the quotient results established in
\cite{AM3}. Recall that, for $f\in\Lip_0(X,Y)$, its linearization
$\widehat{f}\colon\F(X)\longrightarrow Y$
satisfies
$f\in\mathcal L(X,Y)
\quad\Longleftrightarrow\quad
\widehat{f}|_{\ker(\beta_X)}=0$.
Now, by the Remark~\ref{rem:kernel-orbit},
$\ker(\beta_X)=\mathcal Y_X$
where $\mathcal Y_X$ is precisely the orbit-difference subspace
associated with the recentered translation representation $S$.
Hence
\begin{equation}\label{eq:linear-orbit-kernel}
\mathcal L(X,Y)
=
\left\{
f\in\Lip_0(X,Y):
\widehat{f}|_{\mathcal Y_X}=0
\right\}.
\end{equation}

When $Y$ is an injective Banach space, the quotient result from
\cite{AM3}, together with
$\mathcal Y_X=\ker(\beta_X)$, therefore yields the canonical linear
isometric identifications
\begin{equation}\label{eq:quotient-linear-orbit}
\bigslant{\Lip_0(X,Y)}{\mathcal L(X,Y)}
\cong
\mathcal L(\mathcal Y_X,Y)
\cong
\mathcal L(\ker(\beta_X),Y).
\end{equation}

The formulation in \eqref{eq:quotient-linear-orbit} gives the earlier
quotient theorem a natural interpretation in terms of the present
group-action framework.
%================================================
\subsection{Recentered translations and two-Lipschitz mappings}
\label{subsec:two-lipschitz-applications}
%================================================

We now extend the recentered translation construction to
two-Lipschitz mappings. The additive group $X\times Y$ acts
by surjective linear isometries on the corresponding mapping
space, with bounded bilinear mappings as its fixed points.

Let $X$, $Y$, and $E$ be real Banach spaces. Recall that
$\BLipz(X,Y;E)$ consists of all mappings
$T\colon X\times Y\to E$ satisfying
$T(x,0)=T(0,y)=0$ for all $x\in X,\ y\in Y$, and $$
\|T(x,y)-T(x',y)-T(x,y')+T(x',y')\|
\leq C\|x-x'\|\|y-y'\|$$
for some $C\geq 0$ and all $x,x'\in X$, $y,y'\in Y$;
see \cite{TLOI}. We denote the least admissible constant
by $\|T\|_{\operatorname{BLip}}$ and write
$\mathscr B=\BLipz(X,Y;E)$.

\subsubsection{The recentered translation action}

For $a\in X$ and $b\in Y$, we define mappings
$\rho_a^X(T),\rho_b^Y(T)\colon X\times Y\to E$ given by
\begin{equation}\label{eq:coordinate-translations}
\begin{aligned}
\bigl(\rho_a^X(T)\bigr)(x,y)
&=T(x+a,y)-T(a,y),\\
\bigl(\rho_b^Y(T)\bigr)(x,y)
&=T(x,y+b)-T(x,b).
\end{aligned}
\end{equation}
The following proposition shows that these formulas define
operators on $\mathscr B$.

\begin{proposition}\label{prop:two-translation-action}
For every $a\in X$ and $b\in Y$, the operators
$\rho_a^X,\rho_b^Y\colon\mathscr B\longrightarrow\mathscr B$
are surjective linear isometries, which satisfy $\rho_a^X\rho_c^X=\rho_{a+c}^X,~~\rho_b^Y\rho_d^Y=\rho_{b+d}^Y,
~~\rho_a^X\rho_b^Y=\rho_b^Y\rho_a^X$
for all $c\in X, d \in Y$.

Consequently, the operators
$\rho_{(a,b)}(T):=\rho_a^X\bigl(\rho_b^Y(T)\bigr)$
define an isometric representation of the additive group
$X\times Y$ on $\mathscr B$. Explicitly,
\begin{equation}\label{rho-two-lip}
\begin{aligned}
\bigl(\rho_{(a,b)}(T)\bigr)(x,y)
={}&T(x+a,y+b)-T(x+a,b)\\
&-T(a,y+b)+T(a,b).
\end{aligned}
\end{equation}
Here products of operators denote composition on $\mathscr B$.
\end{proposition}

\begin{proof}
Fix $T\in\mathscr B$ and $a\in X$. The mapping
$\rho_a^X(T)$ vanishes on the coordinate axes, since
$\bigl(\rho_a^X(T)\bigr)(0,y)=0$
and
$
\bigl(\rho_a^X(T)\bigr)(x,0)
=T(x+a,0)-T(a,0)=0$.
For a mapping $U\colon X\times Y\to E$, we write
\[
\Delta U(x,x';y,y')
=
U(x,y)-U(x',y)-U(x,y')+U(x',y').
\]
Cancellation of the recentering terms gives us
$\Delta\bigl(\rho_a^X(T)\bigr)(x,x';y,y')
=
\Delta T(x+a,x'+a;y,y')$.
Therefore,
\[
\bigl\|\Delta\bigl(\rho_a^X(T)\bigr)(x,x';y,y')\bigr\|
\leq
\|T\|_{\operatorname{BLip}}
\|(x+a)-(x'+a)\|\|y-y'\|=
\|T\|_{\operatorname{BLip}}
\|x-x'\|\|y-y'\|.
\]
Thus $\rho_a^X(T)\in\mathscr B$ and
$\|\rho_a^X(T)\|_{\operatorname{BLip}}
\leq\|T\|_{\operatorname{BLip}}$.
Interchanging the variables proves the corresponding
assertions for $\rho_b^Y$. Linearity of both operators
follows directly from \eqref{eq:coordinate-translations}.

For $a,c\in X$, we compute
\[
\begin{aligned}
\bigl(\rho_a^X(\rho_c^X(T))\bigr)(x,y)
&=
\bigl(\rho_c^X(T)\bigr)(x+a,y)
-
\bigl(\rho_c^X(T)\bigr)(a,y)\\
&=
\bigl[T(x+a+c,y)-T(c,y)\bigr]-
\bigl[T(a+c,y)-T(c,y)\bigr]\\
&=
T(x+a+c,y)-T(a+c,y)=
\bigl(\rho_{a+c}^X(T)\bigr)(x,y).
\end{aligned}
\]
Hence $\rho_a^X\rho_c^X=\rho_{a+c}^X$. Similarly,
$\rho_b^Y\rho_d^Y=\rho_{b+d}^Y$.
Since every $T\in\mathscr B$ vanishes on the coordinate axes,
$\rho_0^X=\rho_0^Y=I_{\mathscr B}$.
It follows that
$(\rho_a^X)^{-1}=\rho_{-a}^X$, and 
$(\rho_b^Y)^{-1}=\rho_{-b}^Y$.
Now, each operator and its inverse are contractions. In particular,
\[
\|T\|_{\operatorname{BLip}}
=
\|\rho_{-a}^X(\rho_a^X(T))\|_{\operatorname{BLip}}
\leq
\|\rho_a^X(T)\|_{\operatorname{BLip}}
\leq
\|T\|_{\operatorname{BLip}},
\]
so $\rho_a^X$ is an isometry. The same conclusion holds
for $\rho_b^Y$, and both operators are surjective.

Next, applying the definitions successively gives
\[
\begin{aligned}
\bigl(\rho_a^X(\rho_b^Y(T))\bigr)(x,y)
&=
\bigl(\rho_b^Y(T)\bigr)(x+a,y)
-
\bigl(\rho_b^Y(T)\bigr)(a,y)\\
&=
T(x+a,y+b)-T(x+a,b)-T(a,y+b)+T(a,b).
\end{aligned}
\]
Applying the operators in the reverse order gives
\[
\begin{aligned}
\bigl(\rho_b^Y(\rho_a^X(T))\bigr)(x,y)
&=
\bigl(\rho_a^X(T)\bigr)(x,y+b)
-
\bigl(\rho_a^X(T)\bigr)(x,b)\\
&=
T(x+a,y+b)-T(a,y+b)-T(x+a,b)+T(a,b).
\end{aligned}
\]
The two expressions agree. This proves both commutativity
and \eqref{rho-two-lip}.

Finally, for $(a,b),(c,d)\in X\times Y$,
\[
\rho_{(a,b)}\rho_{(c,d)}
=\rho_a^X\rho_b^Y\rho_c^X\rho_d^Y=\rho_a^X\rho_c^X\rho_b^Y\rho_d^Y=\rho_{a+c}^X\rho_{b+d}^Y=\rho_{(a+c,b+d)}.
\]
Moreover, $\rho_{(0,0)}=I_{\mathscr B}$, and each
$\rho_{(a,b)}$ is a surjective linear isometry.
Thus $\rho:=\left(\rho_{(a,b)}\right)_{a \in X,b \in Y}$ is an isometric representation, without any
continuity assumption made in the group.
\end{proof}

%================================================
\subsubsection{Bilinear mappings as fixed points}
%================================================

From \cite{TLOI} we have
$\Blin(X,Y;E)\subseteq\BLipz(X,Y;E)$ with $
\|B\|_{\operatorname{BLip}}=\|B\|$.
The following result is the two-variable counterpart of
Theorem~\ref{thm:linear-fixed-points}.

\begin{theorem}\label{bilinear-fixed}
For the representation $\rho$ constructed in
Proposition~\ref{prop:two-translation-action}, we have
\[
\Fix(\rho)
:=
\{T\in\BLipz(X,Y;E):
\rho_{(a,b)}(T)=T
\text{ for every }(a,b)\in X\times Y\}
=
\Blin(X,Y;E).
\]
\end{theorem}

\begin{proof}
Let $B\in\Blin(X,Y;E)$. Then
\[
\bigl(\rho_{(a,b)}(B)\bigr)(x,y)=
B(x+a,y+b)-B(x+a,b)-B(a,y+b)+B(a,b)=B(x,y).
\]
Thus $B\in\Fix(\rho)$.

Conversely, let $T\in\Fix(\rho)$. Taking $b=0$ and then
$a=0$ in \eqref{rho-two-lip}, and using the vanishing of
$T$ on the coordinate axes, gives
$T(x+a,y)=T(x,y)+T(a,y)$
and
$T(x,y+b)=T(x,y)+T(x,b)$.
Hence $T$ is additive in each variable.
Put $C=\|T\|_{\operatorname{BLip}}$. Now from the two-Lipschitz
estimate we have
\[
\|T(x,y)-T(x',y)\|
\leq C\|x-x'\|\|y\|
~~\mbox{and}~~
\|T(x,y)-T(x,y')\|
\leq C\|x\|\|y-y'\|.
\]
Each partial mapping is therefore continuous and additive,
and hence real-linear. Moreover,
$\|T(x,y)\|\leq C\|x\|\|y\|$
Thus $T$ is bounded bilinear.
\end{proof}

%================================================
\subsubsection{The induced action on the tensor linearization space}
%================================================

Set
$\mathscr F_{X,Y}
=
\F(X)\widehat\otimes_\pi\F(Y)$.
The linearization of two-Lipschitz mappings gives a canonical
surjective linear isometry (see \cite{TLOI}, \cite{AM1})
\[
\BLipz(X,Y;E)
\longrightarrow
\mathcal L(\mathscr F_{X,Y},E),
\qquad T\longmapsto T_L,
\]
where
$T_L(\delta_x\otimes\delta_y)=T(x,y)$ with $
\|T_L\|=\|T\|_{\operatorname{BLip}}$.

Now we consider the recentered translation operators (as in \ref{prop:Sa-action}) on the
Lipschitz-free spaces:
$S_a\delta_x=\delta_{x+a}-\delta_a,$ and $
R_b\delta_y=\delta_{y+b}-\delta_b$.
By Proposition~\ref{prop:Sa-action}, applied to $X$ and $Y$,
respectively, these are surjective linear isometries satisfying
$S_aS_c=S_{a+c},$ and $R_bR_d=R_{b+d}$.
For $(a,b)\in X\times Y$, we define
$\Sigma_{(a,b)}
=
S_a\widehat\otimes_\pi R_b
\colon\mathscr F_{X,Y}\longrightarrow\mathscr F_{X,Y}$.

\begin{proposition}\label{dual-actions}
The family $(\Sigma_{(a,b)})_{(a,b)\in X\times Y}$ is an
isometric representation of the additive group $X\times Y$
on $\mathscr F_{X,Y}$. Moreover, for
$T\in\BLipz(X,Y;E)$,
$\bigl(\rho_{(a,b)}(T)\bigr)_L
=
T_L\circ\Sigma_{(a,b)}$.
In particular, under the scalar-valued identification
$\BLipz(X,Y;\mathbb R)\cong\mathscr F_{X,Y}^*$
we have
$\Sigma_{(a,b)}^*=\rho_{(a,b)}$.
\end{proposition}

\begin{proof}
The projective tensor product construction gives
$\|\Sigma_{(a,b)}\|\leq1$. Furthermore,
\[
\Sigma_{(a,b)}\Sigma_{(c,d)}
=
\Sigma_{(a+c,b+d)},
\qquad
\Sigma_{(0,0)}=I.
\]
In particular, $\Sigma_{(-a,-b)}$ is a contractive inverse
of $\Sigma_{(a,b)}$. Hence every $\Sigma_{(a,b)}$ is a
surjective linear isometry.
Now, for canonical elementary tensors, we have
\begin{equation}\label{Sigma-elementary}
\Sigma_{(a,b)}(\delta_x\otimes\delta_y)
=
(\delta_{x+a}-\delta_a)
\otimes(\delta_{y+b}-\delta_b).
\end{equation}
Consequently,
\[
\begin{aligned}
(T_L\circ\Sigma_{(a,b)})(\delta_x\otimes\delta_y)
&=
T(x+a,y+b)-T(x+a,b)-T(a,y+b)+T(a,b)\\
&=
\bigl(\rho_{(a,b)}(T)\bigr)(x,y).
\end{aligned}
\]
Since the canonical elementary tensors have dense linear
span in $\mathscr F_{X,Y}$, uniqueness of linearization
proves the asserted identity. The scalar-valued adjoint
identity follows immediately.
\end{proof}

%================================================
\subsubsection{The orbit-difference space and its quotient}
%================================================

In accordance with the preceding orbit-difference
constructions, we define
\begin{equation}\label{eq:two-orbit-space}
\mathcal Z_{X,Y}
=
\overline{\operatorname{span}}
\left\{
\Sigma_{(a,b)}u-u:
(a,b)\in X\times Y,\ u\in\mathscr F_{X,Y}
\right\}.
\end{equation}
The closure is taken in the projective tensor norm.
By density and \eqref{Sigma-elementary}, this is equivalently
the closed linear span of
$(\delta_{x+a}-\delta_a)
\otimes(\delta_{y+b}-\delta_b)
-
\delta_x\otimes\delta_y,
$ for any $ x,a\in X,\ y,b\in Y$.

\begin{proposition}\label{tensor-annihilator}
Under the canonical identification
$\mathscr F_{X,Y}^*=\BLipz(X,Y;\mathbb R)$, 
$~~\mathcal Z_{X,Y}^{\perp}
=
\Blin(X,Y;\mathbb R)$.
\end{proposition}

\begin{proof}
A functional $T\in\mathscr F_{X,Y}^*$ annihilates
$\mathcal Z_{X,Y}$ if and only if
$\Sigma_{(a,b)}^*T=T$ for any  $((a,b)\in X\times Y)$.
By Proposition~\ref{dual-actions}, this is equivalent to
$T\in\Fix(\rho)$. The conclusion follows from
Theorem~\ref{bilinear-fixed}.
\end{proof}

The one-variable identity
$\mathcal Y_X=\ker(\beta_X)$ has the following tensor
counterpart. Let
$Q=\beta_X\widehat\otimes_\pi\beta_Y
\colon
\mathscr F_{X,Y}\longrightarrow X\widehat\otimes_\pi Y$.

\begin{theorem}\label{kernel-orbit}
We have
$\mathcal Z_{X,Y}=\ker Q$.
Moreover, $Q$ induces a surjective linear isometry
$\widetilde Q\colon
\bigslant{\mathscr F_{X,Y}}{\mathcal Z_{X,Y}}
\longrightarrow X\widehat\otimes_\pi Y$,
determined by
$\widetilde Q
\bigl(\delta_x\otimes\delta_y+\mathcal Z_{X,Y}\bigr)
=
x\otimes y$.
\end{theorem}

\begin{proof}
The barycentre maps $\beta_X$ and $\beta_Y$ are metric
quotient maps. Since completed projective tensor products
preserve metric quotient maps, $Q$ is a metric quotient
map onto $X\widehat\otimes_\pi Y$. In particular,
$(\ker Q)^\perp=\operatorname{ran}(Q^*)$.

Under the canonical identification
$(X\widehat\otimes_\pi Y)^*
=
\Blin(X,Y;\mathbb R)$,
the adjoint $Q^*$ sends a bounded bilinear form $B$ to the
functional on $\mathscr F_{X,Y}$ satisfying
$(Q^*B)(\delta_x\otimes\delta_y)=B(x,y)$.
Thus its range consists precisely of bounded bilinear
forms viewed as two-Lipschitz forms. Therefore, by
Proposition~\ref{tensor-annihilator}, we have
$(\ker Q)^\perp
=
\Blin(X,Y;\mathbb R)
=
\mathcal Z_{X,Y}^\perp$.
Since $\ker Q$ and $\mathcal Z_{X,Y}$ are norm-closed
subspaces, equality of their annihilators implies
$\ker Q=\mathcal Z_{X,Y}$.

The asserted isometric identification now follows from
the metric quotient property of $Q$.
\end{proof}

\begin{remark}
The relation with the one-variable construction can also
be expressed as
\[
\mathcal Z_{X,Y}
=
\overline{
\ker(\beta_X)\otimes\F(Y)
+
\F(X)\otimes\ker(\beta_Y)
}^{\,\mathscr F_{X,Y}}.
\]
Here the tensor products inside the closure are algebraic
tensor products, viewed inside $\mathscr F_{X,Y}$.
Indeed, this is the kernel formula for the projective
tensor product of the two quotient maps.

Thus the translation quotient imposes the linearity
relations in each tensor factor and recovers
$X\widehat\otimes_\pi Y$.
\end{remark}
The preceding results give the following characterization
of bounded bilinear mappings.

\begin{corollary}\label{cor:bilinear-characterizations}
For $T\in\BLipz(X,Y;E)$, the following are equivalent:
\begin{enumerate}
\item[(i)] $T\in\Blin(X,Y;E)$;
\item[(ii)] $\rho_{(a,b)}(T)=T$ for every $(a,b)\in X\times Y$;
\item[(iii)] $T_L\circ\Sigma_{(a,b)}=T_L$ for every
$(a,b)\in X\times Y$;
\item[(iv)] $T_L|_{\ker Q}=0$.
\end{enumerate}
\end{corollary}

\begin{proof}
The equivalence of \textup{(i)} and \textup{(ii)} is
Theorem~\ref{bilinear-fixed}, and the equivalence of
\textup{(ii)} and \textup{(iii)} follows from
Proposition~\ref{dual-actions}.
Finally, \textup{(iii)} holds precisely when $T_L$
annihilates the closed span of all orbit differences,
which equals $\ker Q$ by Theorem~\ref{kernel-orbit}.
\end{proof}

The orbit quotient therefore has the expected bilinear
universal property.

\begin{corollary}\label{cor:bilinear-universal-property}
There are canonical surjective linear isometries
\[
\Blin(X,Y;E)
\cong
\mathcal L\left(
\bigslant{\mathscr F_{X,Y}}{\mathcal Z_{X,Y}},E
\right)
\cong
\mathcal L(X\widehat\otimes_\pi Y,E).
\]
The first correspondence associates with $B$ the unique
bounded linear operator $\widetilde B$ satisfying
$\widetilde B
\bigl(\delta_x\otimes\delta_y+\mathcal Z_{X,Y}\bigr)
=
B(x,y)$.
Moreover, $\|\widetilde B\|=\|B\|$.
\end{corollary}

\begin{proof}
Combining Theorem~\ref{kernel-orbit} with the universal
property of the completed projective tensor product the result follows.
\end{proof}

%================================================
\subsubsection{Amenable averaging onto the bilinear mappings}
%================================================

The preceding fixed-point characterization gives a
group-theoretic interpretation of the projection constructed
by Mandal in \cite{AM1}.

Let $X$, $Y$, and $E$ be nonzero real Banach spaces, with
$E$ a dual space. Since the additive group $X\times Y$
is abelian, it admits a contractive translation-invariant
$E$-valued mean
$M\colon\ell_\infty(X\times Y,E)\longrightarrow E$
that fixes constant functions. For $T\in\BLipz(X,Y;E)$,
the projection in \cite{AM1} is defined by
$P(T)(x,y)=M(\phi_T^{x,y})$,
where
$
\phi_T^{x,y}(a,b)
=T(a+x,b+y)-T(a+x,b)-T(a,b+y)+T(a,b)$.
This function is bounded, since
$\|\phi_T^{x,y}(a,b)\|
\leq\|T\|_{\operatorname{BLip}}\|x\|\|y\|$.

Then from \eqref{rho-two-lip}, it follows
$\phi_T^{x,y}(a,b)
=
\bigl(\rho_{(a,b)}(T)\bigr)(x,y)$.
Thus the projection admits the orbit-averaging formula
\begin{equation}\label{P-bilinear}
P(T)(x,y)
=
M_{(a,b)}
\left[ (a,b) \mapsto
\bigl(\rho_{(a,b)}(T)\bigr)(x,y)
\right].
\end{equation}
Here the subscript indicates the variable to which the
mean is applied.

We recall the following result from \cite{AM1}.

\begin{theorem}[\cite{AM1}]\label{bilinear-projection}
Let $X$, $Y$, and $E$ be nonzero real Banach spaces,
and suppose that $E$ is a dual space. The operator
$P$ defined by \eqref{P-bilinear} is a norm-one projection
from $\BLipz(X,Y;E)$ onto $\Blin(X,Y;E)$.
\end{theorem}

We refer to \cite{AM1} for the proof. By
Theorem~\ref{bilinear-fixed}, it follows that
$\operatorname{ran}(P)=\Blin(X,Y;E)=\Fix(\rho)$
Consequently, the construction of \cite{AM1} is precisely
invariant-mean averaging onto the fixed-point space of
the recentered translation representation. For the same
choice of mean, the two formulas define the same projection.
% Continue with Theorem~\ref{bilinear-fixed}.

%================================================
\subsection{Bidual realizations of the translation and dilation quotients}
%================================================

The quotient identifications obtained above, combined
with Proposition~\ref{prop:amenable-coinvariants}, yield
a common bidual realization for the three constructions.

\begin{corollary}\label{cor:applications-bidual}
Let $X$ and $Y$ be real Banach spaces. There exist
linear isometric embeddings
$$\F^{ph}(X)\longrightarrow\F(X)^{**},~~X\longrightarrow\F(X)^{**},
~~\mbox{and}~~
X\widehat\otimes_\pi Y
\longrightarrow
\bigl(\F(X)\widehat\otimes_\pi\F(Y)\bigr)^{**}.$$
\end{corollary}

\begin{proof}
For the first assertion, we apply
Proposition~\ref{prop:amenable-coinvariants} to the
normalized dilation representation
$V_r\delta_x=r^{-1}\delta_{rx},$ for any $ r>0$.
Its coinvariant space is $\F^{ph}(X)$.

For the second assertion, we again apply the same proposition
to the recentered translation representation
$S_a\delta_x=\delta_{x+a}-\delta_a$.
Its orbit-difference subspace is $\ker(\beta_X)$,
and its coinvariant space is canonically isometric
to $X$.

Finally, apply the proposition to
$\Sigma_{(a,b)}=S_a\widehat\otimes_\pi R_b$ on
$\F(X)\widehat\otimes_\pi\F(Y)$.
By Theorem~\ref{kernel-orbit}, its coinvariant space
is canonically isometric to $X\widehat\otimes_\pi Y$.
All the acting groups are abelian and are regarded
as discrete.
\end{proof}

\begin{remark}
These embeddings depend on the choice of invariant mean.
They identify the quotient spaces with subspaces of the
corresponding biduals; they do not, by themselves,
identify them with complemented subspaces of the
original linearization spaces.
\end{remark}

\begin{remark}[Compact groups] \label{rem:compact group}
Suppose that a compact group $K$ acts strongly continuously
by surjective linear isometries $V_k$ on a Banach space $X$.
Normalized Haar averaging defines a contractive projection
\[
R_Ku=\int_K V_ku\,d\nu(k)
\]
on $X$. Its kernel is the orbit-difference subspace
$\mathcal D_V$, and consequently
$X/\mathcal D_V\cong\Fix(V)$.

For a continuous point-preserving isometric action on
$M$, this recovers the complemented realization
$\F(M/K)\cong\Fix(V)\subseteq\F(M)$
from \cite[Corollary~4.6]{PILFSIBGA}.

If a compact group acts continuously by similarities
on a nontrivial metric space, its similarity character
is continuous and must be trivial: its image is a
compact subgroup of $(0,\infty)$.
Thus genuinely nontrivial similarity characters require
a setting beyond this compact-group case.
\end{remark}

%================================================
\section{A unified picture}
%================================================

The preceding constructions are instances of a common
fixed-point and orbit-quotient framework. Let $E$ be a
real Banach space and let
\[
V\colon G\longrightarrow\operatorname{Isom}_{\mathrm{lin}}(E)
\]
be a representation. We define
$\mathcal D_V
=
\overline{\operatorname{span}}
\{V_gu-u:g\in G,\ u\in E\}$.
The associated dual representation is
$U_g=V_{g^{-1}}^*$, and the canonical quotient duality gives
\[
(E/\mathcal D_V)^*
\cong
\mathcal D_V^\perp
=
\Fix(U).
\]
Thus the fixed-point space on the dual side has the
coinvariant space $E/\mathcal D_V$ as a canonical predual.
This identification does not require amenability.

For the scalar-valued spaces considered in this article,
the resulting correspondences are summarized below.
Here $\mathscr F_{X,Y}=\F(X)\widehat\otimes_\pi\F(Y)$,
and $M/G$ denotes the metric orbit quotient.

\[
\begin{array}{c|c|c}
\text{Representation on }E
&
\text{Dual fixed-point space}
&
E/\mathcal D_V
\\ \hline
V_g\delta_x=\delta_{gx}
&
\Lip_0^G(M)
&
\F(M/G)
\\[3pt]
V_g\delta_x=\chi(g)^{-1}\delta_{gx}
&
\Lip_0^{G,\chi}(M)
&
\F^{G,\chi}(M)
\\[3pt]
V_r\delta_x=r^{-1}\delta_{rx}
&
\Lip_0^{ph}(X)
&
\F^{ph}(X)
\\[3pt]
S_a\delta_x=\delta_{x+a}-\delta_a
&
X^*
&
X
\\[3pt]
\Sigma_{(a,b)}=S_a\widehat\otimes_\pi R_b
&
\Blin(X,Y;\mathbb R)
&
X\widehat\otimes_\pi Y
\end{array}
\]

In the first two rows, the metric actions fix the
distinguished point. The first row is the special case
$\chi\equiv1$ of the second, while the third corresponds
to the dilation action with $\chi(r)=r$.

The defining relations in these quotients encode the
corresponding conditions on mappings. Invariance and
character-equivariance are expressed, respectively, by
$\delta_{gx}-\delta_x,
~\mbox{and}~\delta_{gx}-\chi(g)\delta_x$.
Positive homogeneity is expressed by
$\delta_{rx}-r\delta_x$,
whereas the recentered translation relations
$\delta_{x+a}-\delta_x-\delta_a$
encode additivity. In the two-variable setting, the
relations
$\Sigma_{(a,b)}u-u$
impose linearity in each variable and yield the completed
projective tensor product.

These identifications also explain the distinction
between metric orbit quotients and the quotients arising
from recentered translations. The translation action of
$X$ on itself has only one ordinary orbit, whereas its
recentered representation on $\F(X)$ has coinvariant
space canonically isometric to $X$.

When $G$ is amenable and regarded as discrete, the averaging
construction of C\'uth and Doucha \cite[Lemma~3.2]{PILFSIBGA},
in the form of Proposition~\ref{prop:amenable-coinvariants},
provides additional structure:
$E/\mathcal D_V\longrightarrow E^{**}$
admits a linear isometric embedding, and $\Fix(U)$ is
the range of a contractive projection on $A^*$. Applied
to the representations above, this accounts for the
bidual realizations of the equivariant free spaces and
the translation quotients, as well as the scalar-valued
averaging projections.

For strongly continuous representations of compact groups,
Haar averaging takes values in $E$ itself and gives $E/\mathcal D_V\cong\Fix(V)\subseteq E$
linearly isometrically, with $\Fix(V)$ contractively
complemented in $E$. In the isometric metric setting,
this recovers the complemented realization of
$\F(M/G)$ established in \cite[Corollary~4.6]{PILFSIBGA}.

The framework therefore separates three conclusions:
the quotient description of the dual fixed-point space,
which is independent of amenability; its realization
inside the bidual through amenable averaging; and a
complemented realization inside the original space,
which requires additional hypotheses.

%================================================
%\noindent\textbf{Declaration of generative AI and AI-assisted technologies in the manuscript preparation process.}
%During the final stage of manuscript preparation, the author used an \textbf{AI assistant}, namely ChatGPT Plus for proofreading and assistance with basic facts and references concerning amenability in Sections~3.4 and~4.4. AI assistance was also used in connecting one of the results of \cite{PILFSIBGA} with the framework developed in the present article, particularly through Remark~\ref{rem:compact group}.
%The author independently verified all mathematical statements and proofs, reviewed and revised the AI-assisted content, and takes full responsibility for the final version of the article.

\noindent\textbf{Acknowledgements}
The author thanks Professor Anil Kumar Karn for discussions during his doctoral studies on some of the Lipschitz-space results discussed here as applications.

The author declares that he has no competing interests.

\noindent\textbf{Data availability.}
No datasets were generated or analysed during this study.

\end{document}